\documentclass[12pt]{article}
\usepackage[usenames,dvipsnames]{color}
\usepackage{graphicx,amsthm,amsmath,amssymb,color,srcltx,xr,xargs,a4wide,bbm}
\usepackage[colorlinks]{hyperref}
\usepackage[shortlabels]{enumitem}

\newcommand\iid{i.i.d.}
\newcommand\wrt{with respect to}
\newcommand\ie{i.e.}
\newcommand\eg{e.g.}

\newcommand\constant{\mathrm{cst}}

\newcommand\Nset{\mathbb{N}}
\newcommand\Zset{\mathbb{Z}}
\newcommand\Rset{\mathbb{R}}

\newcommand\esp{\mathbb E}
\newcommand\pr{\mathbb P}
\newcommand\var{\mathrm{var}}
\newcommand\cov{\mathrm{cov}}
\newcommand{\tail}[1]{\overline{#1}}

\newcommand\ind[1]{\mathbbm{1}{\left\{#1\right\}}}

\newcommand{\bszero}{{\boldsymbol0}}

\newcommand{\vectorbold}[1]{\boldsymbol{#1}}

\newcommand\vbx{\vectorbold{x}}
\newcommand\vbW{\vectorbold{W}}
\newcommand\vbX{\vectorbold{X}}
\newcommand\vby{\vectorbold{y}}

\newcommand\rme{\mathrm{e}}

\newcommand\rmd{\mathrm{d}} 

\newcommand{\scalfunccev}{b}
\newcommand{\scalingexp}{\kappa}

\newcommandx{\tepseq}[1][1=n]{u_{#1}}  

\newcommand{\cevcdfuniv}{\Psi}

\newcommand\tepextrindep{\mathbb{M}} 
\newcommand\TEPextrindep{\mathbb{I}} 
\newcommand\teplimextrindep{{\bbM}}
\newcommand\tedextrindep{\widetilde{M}}
\newcommand\TEDextrindep{\widetilde{I}}
\newcommand\funcextrindep{\mcm}
\newcommand\tepabs{\mathsf{T}}

\newcommand{\norm}[1]{\left|#1\right|} 

\newcommand\numultcond[1]{\boldsymbol{\mu}_{\vectorbold{#1}}}
\newcommandx{\numult}[2][2=]{{\boldsymbol{\nu}}_{\vectorbold{#1}_{#2}}} 
\newcommand\nucond{\boldsymbol{\mu}} 
\newcommand\nucondspectral{\boldsymbol{\sigma}} 

\newcommandx{\dhinterseq}[1][1=n]{r_{#1}}  
\newcommandx{\dhinterseqsmall}[1][1=n]{l_{#1}}  

\newcommand\mbg{\mathbb{G}}
\newcommand\mce{\mathcal E}
\newcommand\mcf{\mathcal F}
\newcommand\mcg{\mathcal G}
\newcommand\mci{\mathcal I}
\newcommand\mcm{\mathcal M}
\newcommand\bbM{\mathbb M}
\newcommand\mcv{\mathcal V}
\newcommand\mcy{\mathcal Y}
\newcommand\mby{\mathbb Y}

\newcommandx{\sequence}[3][2=\Zset,3=j]{\{#1_{#3},#3\in#2\}}

\newcommand\convprob{\stackrel{\mbox{\small\tiny p}}{\longrightarrow}}
\newcommand\convfidi{\stackrel{\mbox{\small\tiny fi.di.}}{\to}} 
\newcommand\convvague{\stackrel{\mbox{\tiny\rm v}}{\longrightarrow}} 
\newcommand\convdistr{\stackrel{\mbox{\tiny\rm d}}{\longrightarrow}} 
\newcommand\convweak{\stackrel{\mbox{\tiny\rm w}}{\longrightarrow}} 
\newcommand\law[1]{\mathcal{L}\left(#1\right)} 

\newcommand{\hatcevcdf}{\widehat{\Psi}}  

\newcommandx\orderstat[3][1=X]{#1_{(#2:#3)}}
\newcommand\statinterseq{k}

\newcommand\bmotion{\mathbb{W}} 
\newcommand\bbridge{\mathbb{B}} 

\newcommand\gauss{\mathbf{N}} 

\newcommand\indep{\dag}  

\usepackage{cleveref}

\newtheorem{theorem}{Theorem}[section]
\newtheorem{lemma}[theorem]{Lemma}

\newtheorem{corollary}[theorem]{Corollary}

\newtheorem{definition}[theorem]{Definition}
\newtheorem{hypothesis}[theorem]{Assumption}
\theoremstyle{remark}
\newtheorem{remark}[theorem]{Remark}
\newtheorem{example}[theorem]{Example}

\crefname{hypothesis}{Assumption}{Assumptions}

\numberwithin{equation}{section}

\begin{document}
\title{Statistical inference for heavy tailed series with extremal independence}

\author{Clemonell Bilayi-Biakana \and Rafa{\l} Kulik\thanks{University of Ottawa} \and Philippe
  Soulier\thanks{Universit\'e  Paris Nanterre}}
\date{}
\maketitle

\begin{abstract}
  We consider stationary time series $\{X_j,j\in\Zset\}$ whose finite dimensional distributions are
regularly varying with extremal independence. We assume that for each $h\geq1$, conditionally on
$X_0$ to exceed a threshold tending to infinity, the conditional distribution of $X_h$ suitably
normalized converges weakly to a non degenerate distribution. We consider in this paper the
estimation of the normalization and of the limiting distribution.
\end{abstract}

\section{Introduction}
\label{sec:introduction}
Let $\sequence{X}$ be a strictly stationary univariate time series.  We say that the time series
$\{X_j\}$ is regularly varying if all its finite dimensional distributions are regularly varying,
i.e.~for each $h\geq0$, there exists a nonzero boundedly finite measure $\numult{h}$ on
$\Rset^{h+1}\setminus\{\vectorbold{0}\}$ infinity, such that
\begin{align}
  \label{eq:mrv-x}
  \frac{\pr\left( x^{-1}(X_0,\dots,X_h) \in \cdot \right)} {\pr(|X_0|>x)} \convvague \numult{h} \; ,
\end{align}
on $\Rset^{h+1}\setminus\{\vectorbold{0}\}$, as $x\to\infty$, where $\convvague$ means vague
convergence. Following \cite{kallenberg:2017}, we say that a measure $\nu$ defined on a complete
separable metric space~$E$ (endowed with its Borel $\sigma$-field) is boundedly finite if $\nu(B)$
for all Borel bounded sets and a sequence of boundedly finite measures $\{\nu_n\}$ is said to
converge vaguely to a measure~$\nu$ if $\nu_n(f)\to\nu(f)$ for all continuous functions with bounded
support. See also \cite{hult:lindskog:2006} who use the terminology of $\mcm_0$-convergence.  Here the
metric space considered is $\Rset^{h+1}\setminus\{\bszero\}$ endowed with the metric
\begin{align*}
  d_h(\vbx,\vby) = \norm{\vbx-\vby} \vee |\norm{\vbx}^{-1} - \norm{\vby}^{-1} | \; ,
\end{align*}
where $\norm{\cdot}$ is an arbitrary norm on $\Rset^{h+1}$. This metric induces the usual topology
and makes $\Rset^{h+1}\setminus\{\bszero\}$ a complete separable space and bounded sets are sets
separated from zero. Moreover, $\Rset^{h+1}\setminus\{\bszero\}$ is still locally compact so this
definition essentially yields the same notion as the classical vague convergence without the need for
compactification at  infinity.

This assumption implies that there exists $\alpha>0$ such that the measure~$\numult{h}$ is
homogeneous of degree $-\alpha$ and the marginal distribution of $X_0$ is regularly varying and
satisfies the balanced tail condition: there exists $p\in[0,1]$ such that
\begin{align*}
  p = \lim_{x\to\infty} \frac{\pr(X_0>x)}{\pr(|X_0|>x)} = 1 - \lim_{x\to\infty} \frac{\pr(X_0<-x)}{\pr(|X_0|>x)} \; .
\end{align*}
Without loss of generality, we assume that $p>0$.

If $h\geq 1$, there exist two fundamentally different cases: either the exponent measure is
concentrated on the axes or it is not. The former case is referred to as extremal independence and
the latter as extremal dependence. In other words, extremal independence means that no two
components can be extremely large at the same time, and extremal dependence means that some pairs
components can be simultaneously extremely large.

In a time series context, we may want to assess the influence of an extreme event at time zero on
future observations. If the finite dimensional distributions of the time series model under
consideration are extremally independent or more generally if the vector $(X_0,X_m,\dots,X_h)$ is
extremally independent for some $h\geq m\geq1$, then, for any Borel set $A$ which is bounded away
from zero in $\Rset^{h-m+1}$ and $y_0>0$,
\begin{align}
  \label{eq:extremal-indep-joint-exceed}
  \lim_{x\to\infty} \frac{\pr(X_0>xy_0, (X_m\dots,X_h) \in xA)}{\pr(|X_0|>x)} = 0  \; .
\end{align}
Thus in case of extremal independence the exponent measure $\numult{h}$ provides
no information on (most) extreme events occurring after an extreme event at
time~0.

In order to obtain a non degenerate limit in~(\ref{eq:extremal-indep-joint-exceed}) and a finer
analysis of the sequence of extreme values, it is necessary to change the normalization
in~(\ref{eq:mrv-x}), and possibly the space on which we will assume that vague convergence
holds. One idea is to find a sequence of normalizations $\scalfunccev_j(x)$, $j\geq1$ such that for
each $h\geq1$, the conditional distribution of
$(X_0/x,X_1/\scalfunccev_1(x),\dots,X_h(x)/\scalfunccev_h(x))$ given $|X_0|>x$ has a non degenerate
limit. Pursuing in the direction opened by \cite{heffernan:resnick:2007} and
\cite{das:resnick:2011}, \cite{lindskog:resnick:roy:2014} and \cite{kulik:soulier:2015} we will
consider vague convergence on the set $(\Rset\setminus\{0\})\times\Rset^h$ endowed with the metric
$d_h^0$ on 
defined by
\begin{align*}
  d_h^0(\vbx,\vby) =  \norm{\vbx-\vby} \vee |x_0^{-1}-y_0^{-1}| \; .
\end{align*}
The bounded sets for this metric are those sets $A$ such that $\vbx=(x_0,\dots,x_h) \in A$ implies
$|x_0|>\epsilon$ for some $\epsilon>0$. Note that under the present definition of vague convergence,
we avoid the pitfalls described in \cite{drees:janssen:2017}.
\begin{hypothesis}
  \label{hypo:conditional-indep-measure}
  There exist scaling functions $\scalfunccev_j$, $j\geq1$ and nonzero measures $\nucond_{0,h}$,
  $h\geq1$, boundedly finite on $(\Rset\setminus\{0\})\times \Rset^h$, $h\geq1$, such that
  \begin{align}
    \label{eq:indep-vague-noncentered}
    \frac1{\pr(|X_0|>x)} \pr \left( \left(\frac{X_0}{x},
        \frac{X_1}{\scalfunccev_{1}(x)}, \cdots, \frac{X_h}{\scalfunccev_{h}(x)}
      \right) \in \cdot \right) \convvague \nucond_{0,h} \; ,
  \end{align}
  on $(\Rset\setminus\{0\})\times\Rset^h$ and for every $y_0>0$, the measures
  $\nucond_{0,h}([y_0,\infty) \times \cdot)$ and $\nucond_{0,h}((-\infty,-y_0) \times \cdot)$ on
  $\Rset^h$ are not concentrated on a hyperplane.
\end{hypothesis}
This assumption does not exclude regularly varying time series with extremal dependence for which
$b_j(x)=x$ for all $j\geq 0$. But our interest will be in extremally independent time series for
which $b_j(x)=o(x)$ for all $j\geq0$. This assumption is fulfilled by many time series, like
stochastic volatility models with heavy tailed noise or heavy tailed volatility, exponential moving
averages and certain Markov chains with regularly varying initial distribution and appropriate
conditions on the transition kernel. See \cite{kulik:soulier:2015}, \cite{mikosch:rezapour:2013}
and~\cite{janssen:drees:2016}.

An important consequence of Assumption~\ref{hypo:conditional-indep-measure} is that the functions
$\scalfunccev_j$, $j\geq 1$ are regularly varying (see \cite[Proposition~1]{heffernan:resnick:2007}
and \cite{kulik:soulier:2015}.)  To put emphasis on the regular variation of the functions
$\scalfunccev_j$, we recall the following definition of \cite{kulik:soulier:2015}.
\begin{definition}[Conditional scaling exponent]
  \label{defi:CSE}
  Under Assumption~\ref{hypo:conditional-indep-measure}, for $h\geq1$, we call the index
  $\scalingexp_h$ of regular variation of the functions $\scalfunccev_h$ the (lag $h$) conditional
  scaling exponent.
\end{definition}
The exponents $\scalingexp_h$, $h\geq1$ reflect the influence of an extreme event at time zero on
future lags. Even though we expect this influence to decrease with the lag in the case of extremal
independence, these exponents are not necessarily monotone decreasing. The measures $\nucond_{0,h}$ also
have some important homogeneity properties: For all Borel sets $A_0\subset\Rset\setminus\{0\}$,
$A_1,\dots,A_h\subset\Rset$,
\begin{align}
  \label{eq:homogeneity-cev}
  \nucond_{0,h} \left( tA_0\times \prod_{i=1}^h t^{\scalingexp_i}A_i \right)  =
  t^{-\alpha} \nucond_{0,h} \left(  \prod_{i=0}^h A_i \right) \; .
\end{align}
Equivalently, for all bounded measurable functions $f$,
\begin{align}
  \label{eq:homogeneity-cev}
  \int_{(\Rset\setminus\{0\})\times\Rset^h} f(t^{-1}x_0, t^{-\scalingexp_1}x_1,\dots,t^{-\scalingexp_h} x_h) \nucond_{0,h}(\rmd \vbx) =
  t^{-\alpha}   \int_{(\Rset\setminus\{0\})\times\Rset^h}  f(\vbx) \nucond_{0,h} (\rmd x) \; .
\end{align}
Cf. \cite[Proposition 1]{heffernan:resnick:2007} and \cite[Lemma 2.1]{kulik:soulier:2015}.  Define
the probability measure $\nucondspectral_h$ on $\{-1,1\}\times\Rset^h$ by
\begin{align*}
  \nucondspectral_h(\{\epsilon\}\times A)
  = \int_{\epsilon u_0>1} \int_A \nucond_{0,h}(\rmd u_0,u_0^{\scalingexp_1}\rmd u_1,\dots,u_0^{\scalingexp_h}\rmd u_h) \; ,
\end{align*}
for $\epsilon\in\{-1,1\}$ and $A$ a Borel subset of $\Rset^h$.  Let $\vbW=(W_0,W_1,\dots,W_h)$ be an
$\Rset^{h+1}$ valued random vector with distribution $\nucondspectral_h$.  Then, for every Borel subsets
$A\subset \left(\Rset\setminus\{0\}\right)\times\Rset^h$, we have
\begin{align}\label{eq:representation-nucond}
  \nucond_{0,h}(A) = \int_{0}^\infty \pr((sW_0,s^{\scalingexp_1} W_1 , \dots , s^{\scalingexp_h} W_h) \in A) \, \alpha s^{-\alpha-1} \rmd s \; .
\end{align}
See \cite[Section~2.4]{kulik:soulier:2015}. Let $Y_0$ be a Pareto random variable with tail index
$\alpha$, independent of $\vbW_h$.
Then, as $x\to\infty$,
\begin{align}
   \label{eq:lim-cevcdf}
  \law{  \left( \frac{X_0}{x}, \frac{X_1}{\scalfunccev_1(x)}, \dots, \frac{X_h}{\scalfunccev_h(x)} \right) \mid |X_0|>x} \convdistr Y_0\vbW_h \; .
\end{align}
In particular, we define for $h>0$ the distribution function $\cevcdfuniv_h$ on $\Rset$:
\begin{align}
  \label{eq:def-cevcdfuniv}
  \cevcdfuniv_h(y) & = \pr(Y_0W_h\leq y) = \lim_{x\to\infty} \pr(X_h \leq \scalfunccev_h(x) y \mid |X_0|>x)  \; ,
\end{align}
for all $y\in\Rset\setminus\{0\}$ since the distribution of $Y_0W_h$ is continuous at all points
except possibly~0.

The goal of this paper is to complement the investigation of this assumption started
in~\cite{kulik:soulier:2015} by providing valid statistical procedures to estimate the conditional
 scaling functions $\scalfunccev_h$, the conditional limiting distributions $\cevcdfuniv_h$ and
scaling exponents $\scalingexp_h$.

\section{Statistical inference}
\label{sec:fclt-unfeasible}
Let $F_0$ be a distribution of $|X_0|$.  All our results we be proved under the following
$\beta$-mixing assumptions.
\begin{hypothesis}
  \label{hypo:basics}
  \begin{enumerate}[{\rm ({A}1)}]
  \item \label{item:beta-mixing} The sequence $\{X_j,j\in\Zset\}$ is $\beta$-mixing with rate
    $\{\beta_n,n\geq1\}$.
  \item\label{item:assumptions-on-sequences} There exist a non decreasing sequence $\tepseq$, non decreasing sequences of integers
    $\dhinterseq$ and $\dhinterseqsmall$ such that
    \begin{align}
    \label{eq:rn}
      \lim_{n\to\infty} \dhinterseqsmall
      & = \lim_{n\to\infty} \dhinterseq = \lim_{n\to\infty} \tepseq = \lim_{n\to\infty} \frac{\dhinterseq}{\dhinterseqsmall} = \infty \; ,  \\
      \label{eq:beta}
      \lim_{n\to\infty} \frac{n}{\dhinterseq} \beta_{\dhinterseqsmall} & = 0 \; , \\
      \label{eq:rnbarFun0}
      \lim_{n\to\infty}\tepseq & =n\bar{F}_0(\tepseq) = \infty  \; , \ \ \lim_{n\to\infty} \dhinterseq\bar{F}_0(\tepseq)=0 \; .
    \end{align}
\end{enumerate}
\end{hypothesis}

\subsection{Non parametric estimation of the limiting conditional distribution}
\label{sec:est-distrib}
In order to define an estimator of $\cevcdfuniv_h$, we must first consider the infeasible
statistic
\begin{align}
  \label{eq:ted-extrindep}
  \TEDextrindep_{h,n}(s,y) & = \frac{1}{n\tail{F_0}(\tepseq)} \sum_{j=1}^{n-h} \ind{|X_j|>\tepseq s,X_{j+h} \leq
    \scalfunccev_{h}(\tepseq)y} \; .
\end{align}
Then, \Cref{hypo:conditional-indep-measure} and the homogeneity property \eqref{eq:homogeneity-cev} imply that for all $s>0$ and $y\in\Rset$,
\begin{align*}
  \lim_{n\to\infty} \esp[  \TEDextrindep_{h,n}(s,y)] &=\lim_{n\to\infty}\frac{n-h}{n}\frac{\pr(|X_0|>\tepseq s,X_h \leq b(\tepseq) y)}{\tail{F}_0(u_n)} \\
  & =\nucond_{0,h}((s,\infty)\times\Rset^{h-1}\times (-\infty,y])= s^{-\alpha} \cevcdfuniv_h(s^{-\scalingexp_h}y) \; .
\end{align*}
We consider weak convergence of the processes $\widetilde\TEPextrindep_{h,n}$ and
$\TEPextrindep_{h,n}$ defined on $(0,\infty)\times\Rset$ by
\begin{align*}
  \widetilde\TEPextrindep_{h,n}(s,y)  & = \sqrt{n\tail{F_0}(\tepseq)} \{\TEDextrindep_{h,n}(s,y) - \esp[\TEDextrindep_{h,n}(s,y)]\} \;, \\
  \TEPextrindep_{h,n}(s,y)  & = \sqrt{n\tail{F_0}(\tepseq)} \{\TEDextrindep_{h,n}(s,y) - s^{-\alpha} \cevcdfuniv_h(s^{-\scalingexp_h}y)\}\; .
\end{align*}

\begin{hypothesis}
  \label{hypo:fclt-infeasible}
  For all $s, t>0$,
    \begin{align}
      \label{eq:conditiondhs}
      \lim_{\ell\to\infty} \limsup_{n\to\infty} \frac{1}{\tail{F}_0(\tepseq)} \sum_{\ell<|j|\leq\dhinterseq} \pr(|X_0|>\tepseq s,|X_j|>\tepseq t) = 0 \; .
    \end{align}
\end{hypothesis}
\begin{hypothesis}
  \label{item:B2}
  There exists $s_0\in (0,1)$ such that
  \begin{align}
    \label{eq:no-bias-orderstat-bivariate}
    \lim_{n\to\infty}\sqrt{n\tail{F}_0(\tepseq)} \sup_{s\geq s_0,y\in\Rset}
    \left|\frac{\pr(|X_0|>\tepseq s,X_h \leq b(\tepseq) y)}{\tail{F}_0(u_n)}  - s^{-\alpha} \cevcdfuniv_h(s^{-\scalingexp_h} y)\right|=0\;.
  \end{align}
\end{hypothesis}

\begin{remark}
  \label{remark:comments}
  An assumptions similar to (\ref{eq:conditiondhs}) is unavoidable. Its purpose is to prove the
  convergence of the intrablock variance in the blocking method and tightness. The present one is
  taken from~\cite{kulik:soulier:wintenberger:2015}. Similar ones have been considered
  in~\cite{rootzen:2009}, \cite{drees:rootzen:2010} and \cite{drees:segers:warchol:2015}.  Some of
  these conditions have been checked directly for extremally dependent time series like GARCH(1,1)
  or ARMA models (see e.g. \cite{drees:2002tail}), or for Markov chains that satisfy a drift
  condition (cf. \cite{kulik:soulier:wintenberger:2015}). This assumption will be checked in
  \Cref{sec:examples} for some specific models. \Cref{item:B2} is unavoidable if one wants to remove
  bias. This will not be discussed in the paper. The condition holds for \textit{some} sequences
  $u_n$.
\end{remark}
Let $\TEPextrindep_h$ be a the Gaussian process on $(0,\infty)\times\Rset$ with covariance
$\cov(\TEPextrindep_h(s,y),\TEPextrindep_h(t,z)) = (s\vee t)^{-\alpha} \Psi_h((s\vee
t)^{-\scalingexp_h}(y\wedge z))$,
$s,t>0$, $y,z\in\Rset$.
We note that
\begin{align*}
  \bmotion(u)=\TEPextrindep_h(1,\Psi_h^{-1}(u))\;,\ \  u\in (0,1)\;,
\end{align*}
is a standard Brownian motion on $(0,1)$.
The following theorem establishes weak convergence of the tail empirical process
and forms the basis for statistical inference on $\cevcdfuniv_h$.  Its proof is given
in~\Cref{sec:proof-of-theo:tep-extrindep}.
\begin{theorem}
  \label{theo:tep-extrindep}
  Let $\sequence{X}$ be a strictly stationary regularly varying sequence such that
  \Cref{hypo:conditional-indep-measure} with extremal independence at all lags. Assume moreover that
  \Cref{hypo:basics,hypo:fclt-infeasible} hold and that the function $\cevcdfuniv_h$ is continuous
  on $\Rset$. Then the process $\widetilde\TEPextrindep_{h,n}$ converges weakly in
  $\ell^{\infty}([s_0,\infty)\times \Rset)$ to~$\TEPextrindep_h$. If moreover \Cref{item:B2} holds, then
  $\TEPextrindep_{h,n}$ converges weakly in $\ell^{\infty}([s_0,\infty)\times \Rset)$
  to~$\TEPextrindep_h$.
\end{theorem}
We now need proxies to replace $u_n$ and $b(u_n)$ which are unknown in order to obtain a feasible
statistical procedure. As usual, $u_n$ will be replaced by an order statistic. To estimate the
scaling functions $\scalfunccev_h$ we will exploit their representations in terms of conditional
mean. Therefore, we need additional conditions.
\begin{hypothesis}
  \label{hypo:moments}
  There exists $\delta>0$ and $s_0>0$ such that
    \begin{align} 
      \lim_{n\to\infty} r_n \{n \tail{F}_0(\tepseq)\}^{-\delta/2} = 0 \; , 
    \end{align}
  \begin{gather}
    \label{eq:psi-second-moment-extrindep}
    \sup_{n\geq1} \frac1{\tail{F}_0(\tepseq)}  \esp\left[
    \left|   \frac{X_{h}}{\scalfunccev_h(\tepseq)}\right|^{2+\delta}\ind{\{X_0> s_0\tepseq\}}\right]  < \infty \; , \\
    \label{eq:conditiondhs-ext-extrindep}
    \lim_{\ell\to\infty} \limsup_{n\to\infty} \frac{1}{\tail{F}_0(\tepseq)}
     \sum_{\ell< |j| \leq \dhinterseq} \esp\left[\frac{|X_{h}|}{\scalfunccev_h(\tepseq)}
    \frac{|X_{j+h}|}{\scalfunccev_h(\tepseq)} \ind{\{X_0> s_0\tepseq\}}\ind{\{X_j> s_0\tepseq\}}  \right]  = 0 \;      , \\
    \label{eq:no-bias-sums}
    \lim_{n\to\infty} \sup_{s\geq s_0} \sqrt{n\tail{F}_0(u_n)}  \left|\frac{\esp\left[\frac{|X_h|}{b_h(u_n)}\ind{|X_0|> u_n s}\right]}{\tail{F}_0(u_n)}
    - \esp[|W_h|] s^{\kappa_h-\alpha} \right|  = 0 \; .
  \end{gather}
\end{hypothesis}
Condition~(\ref{eq:psi-second-moment-extrindep}) requires $\alpha>2$ and implies that the sequence
$(\scalfunccev_h^{-1}(\tepseq)X_h)^2$ is uniformly integrable conditionally on $|X_0|>\tepseq$ and
therefore,
\begin{align}
  \label{eq:mh}
  \lim_{x\to\infty} \esp \left[ \scalfunccev_h^{-i} (x) |X_h|^i  \mid X_0 > x \right] =
  \int_{-\infty}^\infty |y|^i \, \cevcdfuniv_h(\rmd y) =\esp[Y_0^i]\esp[|W_h|^i] < \infty \; ,  \ \ i=1,2 \; .
\end{align}
Since the function
$\scalfunccev_h$ and the limiting distribution $\cevcdfuniv_h$ are defined up to a scaling constant,
we can and will assume without loss of generality that
\begin{align*}
  \esp[|W_h|]=\int_{-\infty}^\infty |y| \, \cevcdfuniv_h(\rmd y) = \int_1^\infty \int_{-\infty}^\infty |x_h| \nucond_{0,h}(\vbx)   \;  .
\end{align*}
Condition (\ref{eq:conditiondhs-ext-extrindep}) is again unavoidable and must be checked for
specific models. Condition~(\ref{eq:no-bias-sums}) is a bias condition which will not be further
discussed.

Set $\statinterseq=n\tail{F}_0(\tepseq)$ and let
$\orderstat[|X|]{n}{1}\leq \orderstat[|X|]{n}{2}\leq \cdots \leq \orderstat[|X|]{n}{n}$ be the order
statistics of $|X_1|,\ldots,|X_n|$.  Define an estimator of $\scalfunccev_h(\tepseq)$ by
\begin{align}
  \label{eq:def-scalingfunc-estimator}
  \widehat\scalfunccev_{h,n} = \frac{1}{\statinterseq} \sum_{j=1}^{n-h} |X_{j+h}|
  \ind{\{|X_j|>\orderstat[|X|]{n}{n-\statinterseq}\}} \; .
\end{align}

\begin{corollary}
  \label{coro:tepindep-semifeasible}
  Let the assumptions of \Cref{theo:tep-extrindep} and \Cref{hypo:moments}
  hold with extremal independence at all lags.  Then
  \begin{multline*}
    \left(\TEPextrindep_{h,n}, \sqrt{k}\left(\frac{\orderstat[|X|]{n}{n-k}}{\tepseq}-1\right) ,
      \sqrt{k}\left( \frac{\widehat \scalfunccev_{h,n}}{\scalfunccev_h(\tepseq)} - 1    \right)  \right) \\
    \convweak \left(\TEPextrindep_h, \alpha^{-1} \bmotion(1), \int_0^1 |\cevcdfuniv_h^{-1}(u)|
      \rmd\bbridge(u) + \alpha^{-1}\scalingexp_h \bmotion(1) \right) \; ,
  \end{multline*}
  where $\bmotion(u) =
  \TEPextrindep_h(1,\cevcdfuniv_h^{-1}(u))$ is a standard Brownian motion and
  $\bbridge(u) = \bmotion(u) - u\bmotion(1)$ is a standard Brownian bridge on $[0,1]$.
\end{corollary}

\begin{remark}
  The moment conditions in \Cref{hypo:moments} may seem to be too restrictive. In fact, we can
  consider a family of estimators $\widehat b_{h,n}(\zeta)$, where $|X_{j+h}|$
  in~\eqref{eq:def-scalingfunc-estimator} is replaced with $|X_{j+h}|^\zeta$ with some
  $\zeta>0$. However, we do not pursue it in this paper.
\end{remark}
Define now the following estimator of $\cevcdfuniv_h$:
\begin{align}
  \hatcevcdf_{h,n}(y) = \frac{1}{\statinterseq} \sum_{j=1}^{n-h}
  \ind{\{X_j>\orderstat{n}{n-\statinterseq}\}} \ind{\{X_{j+h} \leq \widehat\scalfunccev_{h,n} y\}}
  \label{eq:hatphi}
  = \TEDextrindep_{h,n} \left(\frac{\orderstat{n}{n-\statinterseq}}{\tepseq},
    \frac{\widehat\scalfunccev_{h,n}}{\scalfunccev_h(\tepseq)}y\right) \; .
\end{align}
The theory for this estimator is easily obtained by applying \Cref{coro:tepindep-semifeasible} and
the $\delta$-method.

\begin{corollary}
  \label{theo:clt-psi}
  Under the assumptions of \Cref{coro:tepindep-semifeasible} and if the function $\cevcdfuniv_h$ is
  differentiable,
  $\sqrt{\statinterseq}\left\{\hatcevcdf_{h,n} - \cevcdfuniv_h\right\}\convweak \Lambda_h$ in
    $\ell^\infty([s_0,\infty))$, where the process $\Lambda_h$ is defined by
  \begin{align}
    \Lambda_h(y) & = \bbridge(\cevcdfuniv_h(y)) +  y\cevcdfuniv_h'(y) \int_0^1 |\cevcdfuniv_h^{-1}(u)|\rmd \bbridge(u)  \; ,
                \label{eq:def-limit-Lambda}
  \end{align}
  where $\bbridge$ is the standard Brownian bridge.

\end{corollary}
\begin{remark}
  The additional term in the limiting distribution is due to the method of estimation of the
  conditional scaling function. Note that the limiting distribution depends only on $\cevcdfuniv_h$
  and therefore can be used for a Kolmogorov-Smirnov type goodness of fit test of the conditional
  distribution.
\end{remark}

\subsection{Estimation of the conditional scaling exponent}
\label{sec:est-scaling-exponent}
We now consider the estimation of the scaling exponent $\scalingexp_h$. We will use the following
result.
\begin{lemma}
  \label{prop:tail-product-cev}
  Let Assumption~\ref{hypo:conditional-indep-measure} hold and assume moreover that
  \begin{align}
    \label{eq:condition-epsilon-product-cev-better}
    \lim_{\epsilon\to0} \limsup_{x\to\infty}  \frac{ \pr( |X_0X_h|>x\scalfunccev_h(x) , |X_0| \leq \epsilon x) } {\pr(|X_0|>x)} = 0 \; .
  \end{align}
  Then $\esp[|W_h|^{\frac{\alpha}{1+\scalingexp_h}}]<\infty$ and
  \begin{align}
    \label{eq:tail-scaling-exponent}
    \lim_{x\to\infty}\frac{ \pr( |X_0X_h| > x\scalfunccev_h(x) y) } {\pr(|X_0|>x)}
    = \esp[|W_h|^{\frac{\alpha}{1+\scalingexp_h}}] y^{-\frac{\alpha}{1+\scalingexp_h}}  \; .
  \end{align}
\end{lemma}
This is \cite[Proposition~2]{kulik:soulier:2015}, where the finiteness of
$\esp[|W_h|^{\frac{\alpha}{1+\scalingexp_h}}]$ is assumed, but it is easily seen that this is
actually a consequence of~(\ref{eq:condition-epsilon-product-cev-better}). At this moment this is
all we need to state our results but we will need to prove in \Cref{sec:tail-array-sums} a
generalized version of~\Cref{prop:tail-product-cev}; see \Cref{lem:mucondphi-existence}. It must be
noted that Condition~(\ref{eq:condition-epsilon-product-cev-better}) does not hold for an \iid\
sequence. See also \Cref{sec:stoch-volat-proc}.

If (\ref{eq:condition-epsilon-product-cev-better}) holds, then the product $X_0X_h$ has tail index
$\alpha/(1+\kappa_h)$.  Hence, we can suggest the following estimation procedure of the scaling
exponent $\kappa_h$.
\begin{itemize}
\item Let $\gamma=1/\alpha$, where $\alpha$ is the tail index of the sequence $\{|X_j|\}$. Estimate
  $\gamma$ using the Hill estimator $\widehat\gamma$ based on an intermediate sequence $k$, \ie
  \begin{align*}
    \hat\gamma = \frac{1}k \sum_{j=1}^n \log(\orderstat[|X|]{n}{n-j+1}/\orderstat[|X|]{n}{n-k}) \; .
  \end{align*}
\item Let $\gamma_h=(1+\scalingexp_h)\gamma=(1+\scalingexp_h)/\alpha$ be estimated by
  $\widehat\gamma_h$, the Hill estimator of the tail index of $|X_0X_h|$, based on the sequence
  $V_j=|X_jX_{j+h}|$, $j=1,\dots,n$ (assuming without loss of generality that we have $n+h$
  observations) and on the same intermediate sequence:
  \begin{align*}
    \hat\gamma_h = \frac{1}k \sum_{j=1}^n \log(\orderstat[V]{n}{n-j+1}/\orderstat[V]{n}{n-k}) \; .
  \end{align*}
\item Estimate $\scalingexp_h=\gamma_h/\gamma-1$ by
\begin{align}
  \label{eq:scaling-exp-estimator}
  \widehat\scalingexp_h = {\widehat\gamma_h}/{\widehat\gamma} - 1 \; .
\end{align}
\end{itemize}
Asymptotic normality of the Hill estimator for beta-mixing sequences is well known.  See \eg\
\cite{drees:2000,drees:2002tail}. The asymptotic normality of $\widehat\kappa_h$ will follow from
the delta method. To state the result, we need additional anti-clustering and second-order
conditions.

\begin{hypothesis}
  \label{hypo:conditions-S}
  For all $s, t>0$,
  \begin{align}
      \label{eq:conditiondhs-product}
      \lim_{\ell\to\infty} \limsup_{n\to\infty} \frac{1}{\tail{F}_0(\tepseq)}
      \sum_{\ell<|j|\leq\dhinterseq} \pr(|X_0X_h|>\tepseq b_h(u_n) s,|X_jX_{j+h}|>\tepseq b_h(u_n) t) = 0 \; .
    \end{align}
    Furthermore,
    \begin{align}
      \label{eq:condition-slog-univ}
      \lim_{\ell\to\infty} \limsup_{n\to\infty}
      &  \frac{1}{\tail{F}_0(\tepseq)}
        \sum_{j=\ell}^{r_n} \esp[\log_+(|X_0|/\tepseq)\log_+(|X_j|/\tepseq)] = 0 \; , \\
      \label{eq:condition-slog-product}
      \lim_{\ell\to\infty} \limsup_{n\to\infty}
      &  \frac{1}{\tail{F}_0(\tepseq)}
        \sum_{j=\ell}^{r_n} \esp[\log_+(|X_0X_h|/(\tepseq b(\tepseq)))\log_+(|X_jX_{j+h}|/(\tepseq b(\tepseq)))] = 0 \; .
    \end{align}
\end{hypothesis}

\begin{theorem}
  \label{thm:clt-kappa}
  Let $\sequence{X}$ be a strictly stationary regularly varying sequence such
  that~\Cref{hypo:conditional-indep-measure} holds with independence at all lags. Assume moreover
  that \Cref{hypo:basics,hypo:fclt-infeasible,hypo:conditions-S,item:B2} and the bound
  (\ref{eq:condition-epsilon-product-cev-better}) hold and that $k=n\tail{F}_0(\tepseq)$ is chosen
  in such a way that
    \begin{align}
      \label{eq:no-bias-orderstat-bivariate-product}
      \lim_{n\to\infty}\sqrt{k} \sup_{s\geq s_0}\left|
      \frac{\pr(|X_0X_h|>\tepseq b_h(\tepseq)s)} {\overline{F}_0(\tepseq)} - s^{-\alpha/(1+\scalingexp_h)} \right|=0
    \end{align}
    for some $s_0\in(0,1)$.  Then
  \begin{align*}
    \sqrt{k}(\widehat\scalingexp_h-\scalingexp_h) \convdistr
    \gauss\left(0,(1+\kappa_h)\esp[||W_h|^{\frac{\alpha}{1+\scalingexp_h}}-1|]\right)\;.
  \end{align*}
\end{theorem}

\section{Examples}
\label{sec:examples}

\subsection{Stochastic volatility process}
\label{sec:stoch-volat-proc}
Consider the sequence $X_j=\varepsilon_j\exp(Y_j)$, $j\in\Zset$, where $\sequence{Y}$ is a Gaussian
process independent of the \iid\ sequence $\sequence{\varepsilon}$, regularly varying with index
$\alpha$. For simplicity we assume that the random variables $\varepsilon_j$ are nonnegative.  We
list the properties of $X_j$ (see
\cite{davis:mikosch:2001},~\cite{kulik:soulier:2011},~\cite{kulik:soulier:2015}).
\begin{enumerate}[(i),wide=0pt]
\item The sequence $\sequence{X}$ is regularly varying with extremal independence. It satisfies
  Assumption~\ref{hypo:conditional-indep-measure} with $\scalfunccev_h\equiv 1$ for all $h\geq 1$.

\item By Breiman's lemma, $\pr(X_0>x)\sim \esp[\exp(\alpha Y_0)]\pr(\varepsilon_0>x)$ as
  $x\to\infty$.

\item By \cite[Theorem 5.2a),c)]{bradley:2005}, if the spectral density of the Gaussian sequence
  $\sequence{Y}$ is bounded away from zero and if $\cov(Y_0,Y_n)=O(n^{-\delta})$ with $\delta>2$
  then  $\beta(n)=O(n^{2-\delta})$;

\item Conditioning on the sequence ${\mathcal Y}=\sequence{Y}$, the equivalence between the tails of
  $\varepsilon_0$ and $X_0$ and Potter's bounds yield for $\delta>0$,
  \begin{align*}
    \frac{1}{\tail{F_0}(\tepseq)}
    & \sum_{\ell < |j| \leq \dhinterseq} \pr(X_0>\tepseq s,X_j>\tepseq s) \\
    & = \frac{\tail{F}_{\varepsilon}^2(\tepseq)}{\tail{F_0}(\tepseq)} \sum_{\ell<
      |j| \leq \dhinterseq} \esp \left[ \frac{\pr(\varepsilon_0>\tepseq s \exp(-Y_0) \mid\mcy)} {\pr(\varepsilon_0>\tepseq)}
      \frac{\pr(\varepsilon_0>\tepseq s \exp(-Y_j) \mid \mcy)}{\pr(\varepsilon_0>\tepseq)}  \right] \\
    & =  O(\tail{F}_{\varepsilon}(\tepseq)) \sum_{\ell<
      |j| \leq \dhinterseq} \esp\left[\exp((\alpha+\delta)(Y_0+Y_j))\vee 1\right]  = O(\dhinterseq\tail{F_0}(\tepseq)) = o(1) \; ,
  \end{align*}
  as $n\to\infty$ if~(\ref{eq:rnbarFun0}) holds.

\item Fix $\delta>0$.  We again condition on the sequence $\mcy$ and
  apply Potter's bounds:
\begin{align*}
  \frac{1}{\tail{F_0}(\tepseq)}
  & \sum_{\ell<|j| \leq \dhinterseq} \esp\left[|X_hX_{j+h}|\ind{\{X_0>\tepseq s\}}\ind{\{X_j>\tepseq s\}}\right] \\
  & = \frac{\tail{F}_{\varepsilon}^2(\tepseq)}{\tail{F_0}(\tepseq)} (\esp[|\varepsilon_0|])^2 \\
  & \phantom{=} \times \sum_{\ell<|j|\leq\dhinterseq}
    \esp\left[ \rme^{ Y_h}\frac{\pr(\varepsilon_0>\tepseq s \exp(-Y_0) \mid\mcy)}{\pr(\varepsilon_0>\tepseq)}
    \rme^{  Y_{j+h}}\frac{\pr(\varepsilon_0>\tepseq s \exp(-Y_j) \mid \mcy)}  {\pr(\varepsilon_0>\tepseq)}   \right]  \\
  & = O(\tail{F}_{\varepsilon}(\tepseq)) \sum_{\ell<|j|\leq \dhinterseq}
    \esp\left[\exp((Y_h+Y_{j+h}))\left\{\exp((\alpha+\delta)(Y_0+Y_j))\vee 1\right\}\right]\\
  & = O(\dhinterseq\tail{F}_{\varepsilon}(\tepseq))  = o(1) \; ,
 \end{align*}
 whenever (\ref{eq:rnbarFun0}) holds and $\esp\left[|\varepsilon_0|\right]<\infty$.

\end{enumerate}
In summary, the results in \Cref{sec:est-distrib} are applicable to the stochastic volatility model.

On the other hand, condition \eqref{eq:condition-epsilon-product-cev-better} does not hold and hence
the method of estimating the conditional scaling exponent is not applicable here (note however that
the exponent itself is zero).

\subsection{Markov chains}
As in \cite{kulik:soulier:wintenberger:2015}, assume that $\sequence{X}$ is a function of a
stationary Markov chain $\sequence{\mby}$, defined on a probability space $(\Omega,\mcf,\pr)$, with
values in a measurable space $(E,\mce)$. That is, there exists a measurable real valued function
$g:E\to\Rset$ such that $X_j = g(\mby_j)$. Assume moreover that:
\begin{hypothesis}
  \label{hypo:drift-small}
  \begin{enumerate}[(i),wide=0pt]
  \item The Markov chain $\{\mby_j,j\in\Zset\}$ is strictly stationary under $\pr$.
  \item The sequence $\{X_j=g(\mby_j),j\in\Zset\}$ is regularly varying with tail index $\alpha>0$.
  \item The sequence $\{X_j=g(\mby_j),j\in\Zset\}$ satisfies \Cref{hypo:conditional-indep-measure}.
  \item There exist a measurable function $V:E\to[1,\infty)$, $\gamma\in (0,1)$, $x_0\geq1$ and
    $b>0$ such that for all $y\in E$,
    \begin{align}
      \label{eq:drift}
      \esp[V(\mby_1) \mid \mby_0=y] \leq \gamma V(y)  + b\;.
    \end{align}
  \item There exist an integer $m\geq1$ and for all $x\geq x_0$, there exists a probability measure
    $\nu$ on $(E,\mce)$ and $\epsilon>0$ such that, for all $y \in\{V \leq x\}$ and all measurable
    sets $B\in\mce$,
    \begin{align*}
      \pr(\mby_m \in B \mid \mby_0=y) \geq \epsilon \nu(B) \; .
    \end{align*}
  \item There exist $q_0\in(0,\alpha)$ and a constant $c>0$ such that
    \begin{align*}
      |g|^{q_0} \leq c V \; .
    \end{align*}
  \item For every $s>0$,
    \begin{align}
      \limsup_{n\to\infty} \frac{1}{b^{q_0}(u_n)\tail{F}(u_n)} \esp\left[V(\mby_0) \ind{\{g(\mby_0)>u_n s\}}
      \right] < \infty \; ,  \label{eq:def-Q}
    \end{align}
    where $b(x)=b_1(x)$.
  \end{enumerate}
\end{hypothesis}
In \cite{kulik:soulier:wintenberger:2015} we showed that the above assumptions (without (iii) and
with $b(u_n)=u_n$ in \eqref{eq:def-Q}) imply that $\sequence{X}$ is $\beta$-mixing with geometric
rates and the conditions~\eqref{eq:beta},~\eqref{eq:conditiondhs} and
\eqref{eq:psi-second-moment-extrindep}-\eqref{eq:conditiondhs-ext-extrindep} are satisfied.
Following the calculations in \cite{kulik:soulier:wintenberger:2015} we can argue that
\eqref{eq:psi-second-moment-extrindep}-\eqref{eq:conditiondhs-ext-extrindep} hold with
$b_h(u_n)=o(u_n)$.  Therefore, we conclude the following result.
\begin{corollary}
  Assume that \Cref{hypo:drift-small} holds. Assume moreover that the conditions \eqref{eq:rn},
  \eqref{eq:rnbarFun0}, \eqref{eq:no-bias-orderstat-bivariate} are satisfied. Then the conclusion of
  \Cref{theo:tep-extrindep} holds. If also \eqref{eq:no-bias-sums} is satisfied, then the of
  \Cref{coro:tepindep-semifeasible} holds. If moreover $\Psi_h$ is differentiable, then the
  conclusion of \Cref{theo:clt-psi} holds.
\end{corollary}
\begin{example}[Exponential AR(1)]
  Consider $X_{j}=e^{\xi_j}$, $\xi_j=\phi \xi_{j-1}+\varepsilon_j$, where $\phi\in (0,1)$ and
  $\pr(e^{\varepsilon_0}>x)=x^{-\alpha}L(x)$. Then the stationary solution has a regularly varying
  right tail and is tail equivalent to $e^{\varepsilon_0}$,
  cf. \cite{mikosch:rezapour:2013},~\cite{kulik:soulier:2015}. If $\alpha>1$, then
  $\esp[X_1\mid X_0=y]=y^\phi \esp[e^{\varepsilon_0}]$. Hence, the drift function is
  $V(y)=y^\phi$. Condition \eqref{eq:def-Q} holds with $q_0=\phi<\alpha$.
\end{example}

\section{Simulations}\label{sec:simulations}
We simulated from Exponential AR(1) model $X_j=e^{\xi_j}$,
$j=1,\ldots,500$, where $\xi_j=\phi \xi_{j-1}+\epsilon_j$, and
$\epsilon_j$ are i.i.d. with exponential distribution and the
parameter $\alpha$. Hence, $\kappa_1=\phi$, $\kappa_2=\phi^2$, $\kappa_3=\phi^3$.

On Figure 1 we plot estimates of the tail index of $X$ using the Hill estimator along with the confidence intervals:
$$
\widehat\alpha_k\pm 1.96 \frac{1}{\sqrt{k}} \widehat\alpha_k\;, \ \ k=10,\ldots,500\;,
$$
where $\widehat\alpha_k$ is the reciprocal of the Hill estimator based on $k$ order statistics. On
the same graph we plot the estimates of the tail index for products, along with the confidence
intervals (left panel). On the right panels we display estimates of the scaling exponent $\kappa_1$
along with the confidence interval:
$$
\widehat\kappa_1(k)\pm  1.96\frac{1}{\sqrt{k}} (1+\widehat\kappa_1(k))\times \sqrt{{\rm E}[|W_1^{\alpha/(1+\kappa)}-1|]}\;,
$$
where $\widehat\kappa_1(k)$ indicates that the estimator of the scaling exponent is based on $k$
order statistics. The factor $\sqrt{{\rm E}[|W_1^{\alpha/(1+\kappa_1)}-1|]}$ is computed in two
ways. First, note that for our EXPAR(1) we have $W_1={\rm e}^{\epsilon_0}$. Thus, $W_1$ is
exponential with rate $\alpha$. In the first case we plug in known values of $\alpha$ and
$\kappa_1=\phi$ into the expectation and evaluate the multiplicative factor by Monte Carlo
simulation. In the second case, we make the factor depending on $k$, plug-in the estimates
$\widehat\alpha_k$ and $\widehat\kappa_1(k)$ into the expectation and performing Monte Carlo for
each $k$. The first set of confidence intervals is marked in blue, while the second one is plotted
in red.

\begin{figure}
\begin{center}
  \includegraphics[width=0.9\textwidth,height=0.35\textheight]{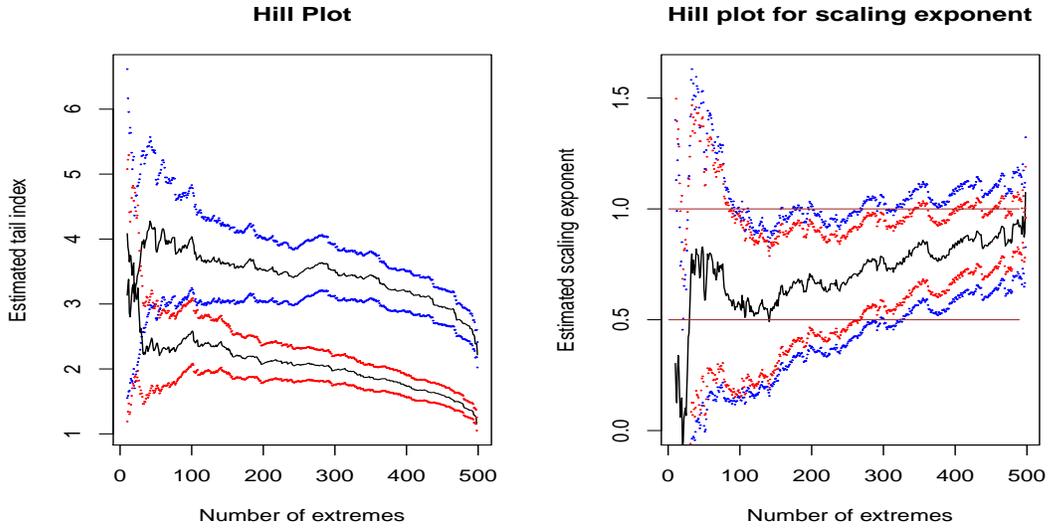}
\end{center}
  \label{fig:expar-sim-2}
  \caption{Exponential AR($1$) model with $\phi=0.5$ and $\alpha=2$. Left panel - tail index
    estimation for the original data and products. Right panel: estimation of $\kappa_1$ along with
    two types of confidence intervals. The horizontal lines indicate the true values and 1, the
    latter indicates extremal dependence.}
\end{figure}
Figure 2 displays boxplots for the estimates of the scaling exponent obtained from 1000 Monte Carlo
simulations, for selected choices of the number of order statistics.

\begin{figure}
\begin{center}
  \includegraphics[width=0.9\textwidth,height=0.3\textheight]{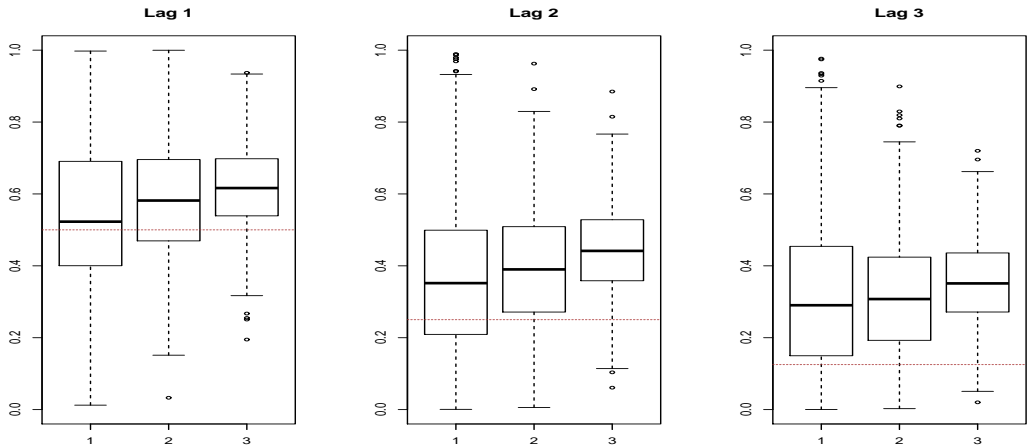}
\end{center}
  \label{fig:exparbox-sim-1}
  \caption{ Exponential AR($1$) model with $\phi=0.5$, $\alpha=4$ and sample size
    $n=500$. Estimation of $\kappa_1$ (left panel), $\kappa_2$ (middle panel) and $\kappa_3$ (right
    panel) for $k=0.05n$ (left box), $k=0.1n$ (middle box), $k=0.2n$ (right box), based on 1000
    repetitions.  }
\end{figure}
\section{Data Analysis}
\label{sec:data}
In this section we apply our theory to the volumes of sales of Microsoft stock prices from January
1, 2010. The data has been detrended by applying simple linear regression.  There is some
correlation in data and the absolute values of residuals. The estimated tail index for residuals is
around 2, while for the products at lag 1, around 1.3-1.4. This indicates extremal independence,
since under extremal dependence we would expect the tail of the product to be 1. The estimate of the
scaling exponent returns 0.6, with the upper confidence interval clearly separated from 1.  The
confidence intervals were calculated under the assumption that the underlying process is EXPAR(1),
described in \Cref{sec:simulations}.

\begin{figure}
\begin{center}
\includegraphics[width=0.45\textwidth,height=0.25\textheight]{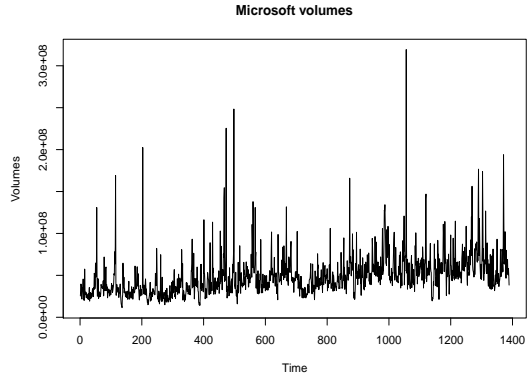}
\end{center}
\caption{Microsoft volumes data}
\end{figure}

\begin{figure}
\begin{center}
\includegraphics[width=0.65\textwidth,height=0.25\textheight]{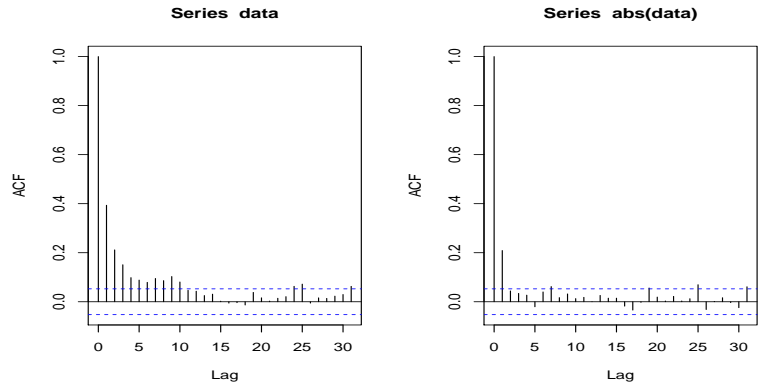}
\end{center}
\caption{Microsoft volumes data: correlations}
\end{figure}

\begin{figure}
\begin{center}
\includegraphics[width=0.65\textwidth,height=0.25\textheight]{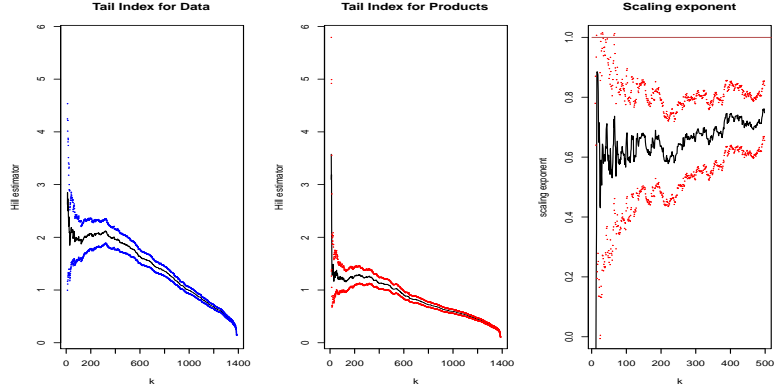}
\end{center}
\caption{Microsoft volumes data: estimation of the tail index (left panel), tail index for products
  (mid-panel) and the conditional scaling exponent for lag 1}
\end{figure}

\section{Proofs}
In this section we prove our results. In~\Cref{sec:tail-array-sums} we prove general results on weak
convergence of tail array sums; see~\Cref{lem:clt-array-fidi,thm:fclt-array}.  Many details are
skipped, since the arguments follow basically the lines of the proofs in
\cite{kulik:soulier:wintenberger:2015}, appropriately modified to incorporate the CEV
assumption. These results will be applied to prove all the results of \Cref{sec:fclt-unfeasible}.

\subsection{Convergence of tail arrays sums}\label{sec:tail-array-sums}
For $0\leq i_1\leq i_2\in\Nset$, $0\leq j_1\leq j_2\in\Nset$ such that $i_2-i_1=j_2-j_1$ denote
\begin{align*}
  \vbX_{i_1,i_2}=(X_{i_1},\ldots,X_{i_2})\;, \ \ \vbX_{i_1,i_2}/{\bf b}_{j_1,j_2}(\tepseq)=(X_{i_1}/b_{j_1}(u_n),\ldots,X_{i_2}/b_{j_2}(u_n))\;.
\end{align*}
Let $\psi:\Rset^{h+1}\to\Rset_+$ be a measurable function such that
\begin{align}
  \label{eq:negligibility-psi}
  &  \lim_{\epsilon\to0} \limsup_{x\to\infty}\frac{\esp\left[|\psi\left(\frac{X_0}{x},\frac{X_1}{b_1(x)},\ldots,
    \frac{X_h}{b_h(x)}\right)|\ind{|X_0|\leq \epsilon x}\right]}{\tail{F}_0(x)} = 0 \; ,
\end{align}
and either $\psi$ is bounded or
there exists $\delta>0$ such that
\begin{align}
    \label{eq:psi-second-moment-1-extrindep}
    \sup_{n\geq1} & \frac{ \esp\left[ \psi^{2+\delta}\left(\frac{X_0}{\tepseq},\frac{X_1}{b_1(\tepseq)},\ldots,
    \frac{X_h}{b_h(\tepseq)}\right)\right]}{\tail{F}_0(\tepseq)} < \infty \; .
  \end{align}
Furthermore, we need a version of the anticlustering condition:
\begin{align}
  \label{eq:condition-sum-phi}
  &  \lim_{\ell\to\infty} \limsup_{n\to\infty} \frac{1}{\tail{F}_0(\tepseq)} \sum_{\ell<|j|\leq\dhinterseq}
    \esp\left[\psi\left(\frac{\vbX_{0,h}}{{\bf b}_{0,h}(\tepseq)}\right)\psi\left(\frac{\vbX_{j,j+h}}{{\bf b}_{0,h}(\tepseq)}\right)\right] = 0 \; ,
\end{align}
where $r_n$ is the sequence from~\ref{item:assumptions-on-sequences} of~\Cref{hypo:basics}.

\begin{definition}
By
$\funcextrindep_\psi$ we denote the linear space of bounded functions $\phi:\Rset^{h+1}\to\Rset$ such that:
\begin{itemize}
\item
 $|\phi| \leq \constant\; \psi$ where $\constant$ depends on $\phi$;
\item
for all $j\geq0$, the function $\vbx_{0,j+h}\mapsto\phi(\vbx_{j,j+h})$ is almost surely
  continuous with respect to $\numultcond{0,j+h}$.
\end{itemize}
\end{definition}
In this section we are interested in convergence of the tail array sums of the form
\begin{align*}
  \tedextrindep_{h,n}(\phi) = \frac{1}{n\tail{F}_0(\tepseq)}
  \sum_{j=1}^{n-h}\phi\left(\frac{X_j}{\tepseq},\frac{X_{j+1}}{b_1(\tepseq)},\ldots,\frac{X_{j+h}}{b_h(\tepseq)}\right)\;.
\end{align*}
We consider finite dimensional convergence of the process
\begin{align*} \widetilde{\tepextrindep}_{h,n}(\phi)=\sqrt{n\tail{F}_0(\tepseq)}\left\{\tedextrindep_{h,n}(\phi)-\esp[\tedextrindep_{h,n}(\phi)]\right\}
    \index{$\tepextrindep_{h,n}$}
\end{align*}
indexed by the set $\funcextrindep_{\psi}$.
\begin{theorem}
  \label{lem:clt-array-fidi}
  Let $\sequence{X}$ be a strictly stationary regularly varying sequence such
  that~\Cref{hypo:conditional-indep-measure} holds with extremal independence at all lags
  and that~\Cref{hypo:basics} is satisfied.
  Let $\psi$ be a measurable function such that  \eqref{eq:negligibility-psi},~\eqref{eq:condition-sum-phi} hold and either $\psi$ is bounded or
  there exists
  $\delta\in (0,1]$ such that~\eqref{eq:psi-second-moment-1-extrindep} and
  \begin{align}
    \lim_{n\to\infty} \frac{\dhinterseq }{\left(n\tail{F}_0(\tepseq)\right)^{{\delta/2}}} = 0 \;
     \label{eq:rate-condition-fidi-unbounded}
  \end{align}
  hold. Then
  \begin{align*}
\left\{\widetilde\tepextrindep_{h,n}(\phi)=\sqrt{n\tail{F}_0(\tepseq)}\left\{\tedextrindep_{h,n}(\phi)-\esp[\tedextrindep_{h,n}(\phi)]\right\},\phi\in\funcextrindep_{\psi}\right\}\convfidi
    \left\{\tepextrindep_h(\phi),\phi\in\funcextrindep_{\psi}\right\}\;,
  \end{align*}
  where $\tepextrindep_h$ is a Gaussian process indexed by $\funcextrindep_{\psi}$ with covariance function
  $(\phi,\varphi)\mapsto\nucond_{0,h}(\phi\varphi)$.
\end{theorem}
In~\Cref{lem:mucondphi-existence} we will justify that under~\eqref{eq:negligibility-psi} the limit
\begin{align*}
  \nucond_{0,h}(\phi) = \lim_{n\to\infty}\esp[\tedextrindep_{h,n}(\phi)]
\end{align*}
is finite for $\phi\in\funcextrindep_{\psi}$. This allows us to consider weak convergence of the process
\begin{align*}
  \tepextrindep_{h,n}(\phi)=\sqrt{n\tail{F}_0(\tepseq)}\left\{\tedextrindep_{h,n}(\phi)-\nucond_{0,h}(\phi)\right\}
\end{align*}
indexed by a subset of $\funcextrindep_{\psi}$. Let $\mcg\subseteq \funcextrindep_{\psi}$ be  equipped  with a
semi-metric $\rho_h$.  The following result mimics Theorem 2.4
in~\cite{kulik:soulier:wintenberger:2015} which in turn is an adaptation
of~\cite[Theorem~2.11.1]{vandervaart:wellner:1996}. Hence, it is stated without a proof.

\begin{theorem}
  \label{thm:fclt-array}
  Let $\sequence{X}$ be a strictly stationary regularly varying sequence such
  that~\Cref{hypo:conditional-indep-measure} holds with extremal independence at all lags. Suppose
  that assumptions of \Cref{lem:clt-array-fidi} are satisfied.
  If moreover
  \begin{enumerate}[(i)]
  \item\label{item:pointwise-separable} $\mcg$ is pointwise separable;
  \item \label{item:envelope-our-entropy} the envelope function
    $\Phi_\mcg=\sup_{\phi\in\mcg} |\phi|$ is in $\funcextrindep_{\psi}$;
  \item \label{item:VC-subgraph} $\mcg$ is a  VC subgraph class or a finite union of such classes;
  \item \label{item:totally-bounded}$(\mcg,\rho_h)$ is totally bounded;
  \item \label{item:continuity} for every sequence $\{\delta_n\}$ which decreases to zero,
    \begin{align}
      \label{eq:continuite-pour-notre-theorem}
      \limsup_{n\to\infty} \sup_{\phi,\varphi\in\mcg \atop \rho_h(\phi,\varphi)\leq \delta_n}
      \frac{\esp\left[\left\{\phi\left(\frac{\vbX_{0,h}}{{\bf b}_{0,h}(\tepseq)}\right)-
      \varphi\left(\frac{\vbX_{0,h}}{{\bf b}_{0,h}(\tepseq)}\right)\right\}^2\right]}{\tail{F}_0(u_n)} = 0  \; ,
    \end{align}
  \end{enumerate}
  then
  \begin{align*}
    \widetilde\tepextrindep_{h,n}(\phi)=\sqrt{n\tail{F}_0(\tepseq)}\left\{\tedextrindep_{h,n}(\phi)-\esp[\tedextrindep_{h,n}(\phi)]\right\}\Rightarrow \tepextrindep_h(\phi)
  \end{align*}
  in $\ell^\infty(\mcg)$. If moreover
  \begin{align}
    \label{eq:second-order-psi}
    \lim_{n\to\infty}\sup_{\phi\in\mcg}\sqrt{n\tail{F}_0(\tepseq)}|\esp[\tedextrindep_{h,n}(\phi)]-\nucond_{0,h}(\phi)|=0\;,
  \end{align}
  then
  \begin{align*}
    \tepextrindep_{h,n}(\phi)=\sqrt{n\tail{F}_0(\tepseq)}\left\{\tedextrindep_{h,n}(\phi)-\nucond_{0,h}(\phi)\right\}\Rightarrow \tepextrindep_h(\phi)
  \end{align*}
  in $\ell^\infty(\mcg)$.
\end{theorem}

The proof of \Cref{lem:clt-array-fidi} will be prefaced by several lemmas.  For brevity, write
\begin{align*}
  V_{j,n}(\phi) & = \phi\left(\frac{X_j}{\tepseq},\frac{X_{j+1}}{b_1(\tepseq)},\ldots,\frac{X_{j+h}}{b_h(\tepseq)}\right) \; , \ \
                  S_n(\phi)  = \sum_{j=1}^{r_n}V_{j,n}(\phi) \; .
\end{align*}

\begin{lemma}
  \label{lem:mucondphi-existence}
  Let~\Cref{hypo:conditional-indep-measure} hold with extremal independence at all lags.  Let $\psi$
  be a measurable function such that~\eqref{eq:negligibility-psi} holds and either $\psi$ is bounded
  or \eqref{eq:psi-second-moment-1-extrindep} holds.  Then, for all $\phi\in\funcextrindep_\psi$ we
  have $\nucond_{0,h}(\phi^2)<\infty$ and
  \begin{align}
    \label{eq:mucondh-phi-limit}
    \lim_{n\to\infty} \frac{\esp[V_{0,n}^2(\phi)] } {\pr(|X_0|>\tepseq)}
    = \numultcond{0,h}(\phi^2) =\int_0^\infty \esp[\phi^2(sW_0,s^{\scalingexp_{1}}W_1,\ldots,s^{\scalingexp_h}W_h)] \alpha s^{-\alpha-1} \rmd s\; .
  \end{align}
  Moreover, for all $j\not=0$ and $\phi,\varphi\in \funcextrindep_{\psi}$,
\begin{align*}
  \lim_{n\to\infty}\frac{1}{\tail{F}_0(\tepseq)}\esp[V_{0,n}(\phi)V_{j,n}(\varphi)] = 0 \; .
\end{align*}
\end{lemma}
\begin{proof}
The proof of the first part is similar to \cite[Proposition~2]{kulik:soulier:2015}.

Assume that $\psi$ is bounded. For
  $\epsilon>0$, we write
\begin{align*}
  \frac{1}{\tail{F}_0(u_n)}\esp[|V_{0,n}^2(\phi)|]&=\frac{1}{\tail{F}_0(u_n)}\esp[|V_{0,n}^2(\phi)|\ind{|X_0|>\epsilon \tepseq}]+
  \frac{1}{\tail{F}_0(u_n)}\esp[|V_{0,n}^2(\phi)|\ind{|X_0|\leq\epsilon \tepseq}]\;.
\end{align*}
By vague convergence and boundedness of $\phi$ the first expression on the right hand side converges
to $\numultcond{0,h}(\phi^2\ind{|y_0|>\epsilon})<\infty$. Application of~\eqref{eq:negligibility-psi}
implies that $\numultcond{0,h}(\phi^2)$ is finite.

If $\psi$ is unbounded, then
then for all $A>0$,
  applying Markov and H\"older inequalities, we obtain
  \begin{align*}
    \lim_{A\to\infty}    \limsup_{n\to\infty} \frac{\esp[V_{0,n}^2(\phi)\ind{|V_{0,n}(\phi)|>A}]}{\tail{F}_0(u_n)}
    \leq \lim_{A\to\infty}   \constant \cdot A^{-\delta/2} \sup_{n\geq1} \frac{\esp[|V_{0,n}(\phi)|^{2+\delta}]}{\tail{F}_0(u_n)} = 0 \; .
  \end{align*}
Thus, we can split $V_{0,n}(\phi)$ as $V_{0,n}(\phi)\ind{|V_{0,n}(\phi)\leq A|}+V_{0,n}(\phi)\ind{|V_{0,n}(\phi)> A|}$ and apply the truncation argument.

As for the second part, thanks to the truncation argument, we can consider bounded functions $\phi,\varphi\in\funcextrindep_{\psi}$. We have
\begin{align}
  \lim_{n\to\infty}\frac{\esp[V_{0,n}(\phi)V_{j,n}(\varphi)]}{\tail{F}_0(u_n)}
  & =     \lim_{n\to\infty}\frac{\esp[V_{0,n}(\phi)V_{j,n}(\varphi)\ind{|X_0|\wedge |X_j|>\epsilon \tepseq}]}{\tail{F}_0(u_n)}
    \label{eq:existence-chj-1} \\
  & \phantom{ = } + \lim_{n\to\infty}\frac{\esp[V_{0,n}(\phi)V_{j,n}(\varphi)\ind{|X_0|\wedge |X_j|\leq \epsilon \tepseq}]}{\tail{F}_0(u_n)} \; .
    \label{eq:existence-chj-2}
\end{align}
The term in~\eqref{eq:existence-chj-1} vanishes, for each $\epsilon>0$, by boundedness of
$\varphi,\phi$ and extremal independence.  The term in~\eqref{eq:existence-chj-2} is bounded by
\begin{align*}
\lim_{n\to\infty}\frac{1}{\tail{F}_0(u_n)}\left\{
\|\varphi\|_{\infty}\esp[V_{0,n}(\phi)\ind{|X_0|\leq \epsilon \tepseq}]+
\|\phi\|_{\infty}\esp[V_{0,n}(\varphi)\ind{|X_0|\leq \epsilon \tepseq}]\right\}
\end{align*}
and hence vanishes as $\epsilon\to 0$ by~\eqref{eq:negligibility-psi}.
\end{proof}

\begin{lemma}\label{lem:cov-existence}
  Let~\Cref{hypo:conditional-indep-measure} hold with extremal independence at all lags and
  \Cref{hypo:basics} hold.  Let $\psi$ be a measurable function such
  that~\eqref{eq:negligibility-psi}, \eqref{eq:condition-sum-phi} hold and either $\psi$ is bounded
  or \eqref{eq:psi-second-moment-1-extrindep} holds.  Then, for all
  $\phi,\varphi\in\funcextrindep_\psi$,
  \begin{align*}
    \lim_{n\to\infty}
    \frac{ \esp\left[ S_n(\phi)S_n(\varphi) \right]}{\dhinterseq\tail{F}_0(\tepseq)}  & = \nucond_{0,h}(\phi\varphi) \; ,
  \end{align*}
  and
  \begin{align*}
    \lim_{n\to\infty}
    \frac{ \cov\left( S_n(\phi),S_n(\varphi) \right)}{\dhinterseq\pr(|X_0|>\tepseq)}  = \numultcond{0,h}(\phi\varphi)\;.
  \end{align*}
\end{lemma}

\begin{proof}
  By stationarity we can write, for $\ell\geq1$,
  \begin{align*}
    \frac{ \esp [ S_n(\phi)S_n(\varphi)]}{r_n\tail{F}_0(u_n)}
    = \sum_{j=-\ell}^\ell \left(1-\frac{|j|}{r_n}\right) \frac{\esp[V_{0,n}(\phi)V_{j,n}(\varphi)]}{\tail{F}_0(u_n)}
    + O \left(\sum_{j=\ell+1}^{r_n}  \frac{\esp[V_{0,n}(\phi)V_{j,n}(\varphi)]}{\tail{F}_0(u_n)} \right)   \; .
  \end{align*}
  Since $V_{0,n}(\phi) V_{0,n}(\varphi) = V_{0,n}(\phi\varphi)$, \Cref{lem:mucondphi-existence}
  shows that the term on the right hand side converges to $\nucond_{0,h}(\phi\varphi)$. The second
  term vanishes by assumption~\eqref{eq:condition-sum-phi}, upon letting $n\to\infty$ and then
  $\ell\to\infty$.

  Finally, by~\Cref{hypo:basics},
  \begin{align*}
    \frac{1}{\tail{F}_0(u_n)}\sum_{j=1}^{\dhinterseq} \esp[V_{0,n}(\phi)]\esp[V_{j,n}(\varphi)]
    =\left(\frac{\esp[V_{0,n}(\phi)]}{\tail{F}_0(u_n)}\right)^2\dhinterseq\tail{F}_0(u_n)\to 0
  \end{align*}
  and hence the result for the covariances follows.
\end{proof}
The next result can be proven along the same lines as of
\cite[Lemmas~3.6-3.7]{kulik:soulier:wintenberger:2015}. In case of unbounded functions, we need
additionally \eqref{eq:rate-condition-fidi-unbounded}.

\begin{lemma}
  \label{lem:negligibility-extrindep}
  Let~\Cref{hypo:conditional-indep-measure} hold with extremal independence at all lags and
  \Cref{hypo:basics} hold.  Let $\psi$ be a measurable function such that
  \eqref{eq:negligibility-psi},~\eqref{eq:condition-sum-phi} hold and either $\psi$ is bounded or
  there exists $\delta\in (0,1]$ such that~\eqref{eq:psi-second-moment-1-extrindep}
  and~\eqref{eq:rate-condition-fidi-unbounded} hold.  Then
  \begin{align}
    \lim_{n\to\infty} \frac{ \esp\left[S_n^2(\phi)\ind{|S_n(\phi)|> \delta \sqrt{n\tail{F}_0(\tepseq)}}\right]}
    {\dhinterseq\tail{F}_0(\tepseq)} = 0  \; .
  \end{align}
\end{lemma}

\begin{proof}[Proof of \Cref{lem:clt-array-fidi}]
  Define $m_n=[n/\dhinterseq]$ and let $\{X_{n,i}^\indep, 1 \leq i \leq m_n,n\geq1\}$ be a
  triangular array of random variables such that the blocks
  $\{X_{(i-1)\dhinterseq+1}^\indep,\dots,X_{i\dhinterseq}^\indep\}$ are independent and each have
  the same distribution as the original stationary blocks, \ie\ the same distribution as
  $(X_1,\dots,X_{\dhinterseq})$ by stationarity of the original sequence. For $i=1,\ldots,m_n$,
  define
  \begin{align*}
    V_{j,n}^\indep(\phi)
    & = \phi\left(\frac{X_j^\indep}{\tepseq},\frac{X_{j+1}^\indep}{b_1(\tepseq)},\ldots,\frac{X_{j+h}^\indep}{b_h(\tepseq)}\right) \; , \ \
    S_{n,i}^\indep(\phi)  = \sum_{j=(i-1)r_n+1}^{ir_n}V_{j,n}^\indep(\phi)\;,
  \end{align*}
  Arguing as in the proof of \cite[Theorem~2.8]{drees:rootzen:2010} or
  \cite{kulik:soulier:wintenberger:2015}, the $\beta$-mixing property and the rate
  condition~(\ref{eq:beta}) implies that it suffices to prove weak convergence of the process
  \begin{align*}
    \widetilde{\tepextrindep}_{h,n}^\indep(\phi)=\{n\tail{F}_0(\tepseq)\}^{-1/2}\sum_{i=1}^{m_n}
    \left\{S_{n,i}^\indep(\phi)-\esp[S_{n,i}^\indep(\phi)]\right\}\;,
  \end{align*}
  along with the appropriate bias condition.  In the first step we show existence of the limiting
  covariance. For $\phi,\varphi$ we have
  \begin{align*}
    \frac{1}{\tail{F}_0(\tepseq)} \sum_{j=1}^{r_n} \esp[V_{0,n}(\phi)] \esp[V_{j,n}(\varphi)]
    = \left(\frac{\esp[V_{0,n}(\phi)]}{\tail{F}_0(\tepseq)}\right) \left(\frac{\esp[V_{0,n}(\varphi)]}{\tail{F}_0(\tepseq)}\right)
    r_n\tail{F}_0(\tepseq) \to 0 \; ,
  \end{align*}
  by~\eqref{eq:mucondh-phi-limit} and~\eqref{eq:rnbarFun0}. Application of~\Cref{lem:cov-existence}
  yields the limiting covariance:
  \begin{align*}
    \lim_{n\to\infty}\cov\left(\widetilde{\tepextrindep}_{h,n}^\indep(\phi),\widetilde{\tepextrindep}_{h,n}^\indep(\varphi)\right)
    =\nucond_{0,h}(\phi\varphi)\;.
  \end{align*}
  \Cref{lem:negligibility-extrindep} finishes the proof.
\end{proof}

\subsection{Proof of \Cref{theo:tep-extrindep}}
\label{sec:proof-of-theo:tep-extrindep}
We only consider the case of extremal independence, since the extremally dependent case can be
concluded directly from \cite{kulik:soulier:wintenberger:2015}. Let $s_0\in (0,1)$. We apply the
results of~\Cref{sec:tail-array-sums} to the function $\psi(x_0,\ldots,x_h)=\ind{|x_0|>s_0}$ and the
class $\mcg_0=\{\phi_{s,y}:(x_0,\ldots,x_h)\to \ind{|x_0|>s,x_h\leq y},s\geq s_0,y\in\Rset\}$.  Then
\begin{align}
  \label{eq:I-M}
  \widetilde\TEPextrindep_{h,n}(s,y) = \widetilde\tepextrindep_{h,n}(\phi_{s,y})\;, \ \ \TEPextrindep_{h,n}(s,y) = \tepextrindep_{h,n}(\phi_{s,y})\;.
\end{align}
We need to check assumptions~\eqref{eq:negligibility-psi}-\eqref{eq:condition-sum-phi}:
\begin{itemize}
\item Condition~\eqref{eq:negligibility-psi} trivially holds since its right hand side vanishes
  whenever $\epsilon\in (0,s_0)$;
\item Condition~\eqref{eq:condition-sum-phi} for the function $\psi$ is implied by the
  anticlustering condition~\eqref{eq:conditiondhs} of \Cref{hypo:fclt-infeasible};
\end{itemize}
Since~\Cref{hypo:conditional-indep-measure,hypo:basics} are already assumed
in~\Cref{theo:tep-extrindep}, by~\Cref{lem:clt-array-fidi}, the finite dimensional distributions of
\begin{align*}
  \sqrt{n\tail{F_0}(\tepseq)} \{\TEDextrindep_{h,n}(s,y) - \esp[\TEDextrindep_{h,n}(s,y)]\}
\end{align*}
converge to those of ${\TEPextrindep}_h$.

To prove tightness, we apply~\Cref{thm:fclt-array}. Condition~\eqref{eq:second-order-psi} reduces
to~\eqref{eq:no-bias-orderstat-bivariate}. It remains to
verify~\ref{item:pointwise-separable}-\ref{item:continuity}. Define the semi-metric $\rho_h$ on
$(0,\infty)\times\Rset$ by
\begin{align}
  \label{eq:metric}
  \rho_h((s,y),(t,z)) = |s-t| + |\cevcdfuniv_h(y)-\cevcdfuniv_h(z)|  \; .
\end{align}
This also defines the semi-metric on $\mcg_0$.  The class is clearly pointwise separable and the
envelope function is $\ind{(s_0,\infty)\times\Rset^h} \in\funcextrindep_{\psi}$.  Also, the class of
indicators of the sets $(s,\infty)\times (-\infty,y]$ is a VC class of index 2.  The proof of
\ref{item:totally-bounded}-\ref{item:continuity} follows the same lines as that of \cite[Theorem
2.11]{kulik:soulier:wintenberger:2015}.  It remains to identify the limiting Gaussian process. For
$0<s<t$ and $y,z\in\Rset$, we have
\begin{align*}
  \cov({\tepextrindep}(\phi_{s,y}),{\tepextrindep}(\phi_{t,z}))
  & =\nucond_{0,h}(\phi_{s,y}\phi_{t,z}) =  t^{-\alpha} \cevcdfuniv_h(t^{-\scalingexp_h}(y\wedge z)) \; .
\end{align*}
This finishes the proof.

\subsection{Central limit theorem for the conditional mean}
\label{sec:cond-mean}
We state another corollary to \Cref{thm:fclt-array}.  For $s>0$, set
$\psi_s(x_0,\ldots,x_h)=|x_h|\ind{|x_0|>s}$ and define the class
\begin{align*}
  \mcg_1=\{\psi_s,s\geq s_0\}\;.
\end{align*}
Then,
\begin{align}
  \label{eq:convergence-sums}
  &\tepextrindep_{h,n}(\psi_{s})=\sqrt{n\tail{F}_0(u_n)}\left\{\frac{1}{n\tail{F}_0(u_n)} \sum_{j=1}^{n-h} \frac{|X_{j+h}|}{b_h(u_n)}
    \ind{|X_j|> u_n s}-\nucond_{0,h}(\psi_s)\right\} \; .
\end{align}
By the homogeneity property \eqref{eq:homogeneity-cev} and the assumption
$\int_{-\infty}^\infty |y| \, \cevcdfuniv_h(\rmd y) =1$, we have
\begin{align}
  \nucond_{0,h}(\psi_s) = s^{\scalingexp_h-\alpha} \; .
\label{eq:nucondpsi_s}
\end{align}

\begin{corollary}
  \label{cor:sums}
  Let $\sequence{X}$ be a strictly stationary regularly varying sequence such
  that~\Cref{hypo:conditional-indep-measure} holds with extremal independence at all lags and
  that~\Cref{hypo:basics} is satisfied. Assume that \Cref{hypo:fclt-infeasible,item:B2,hypo:moments}
  and \eqref{eq:rate-condition-fidi-unbounded} hold.  Then
  \begin{align*}
    \tepextrindep_{h,n}    \convweak     \teplimextrindep_h \; , \ \mbox{ in } \ell^\infty(\mcg_0\cup \mcg_1) \; .
  \end{align*}
\end{corollary}

\begin{proof}
  Let $s_0\in (0,1)$. We apply the results of~\Cref{sec:tail-array-sums}.  We need to check the
  assumptions \eqref{eq:negligibility-psi}-\eqref{eq:condition-sum-phi}:
  \begin{itemize}
  \item Condition \eqref{eq:negligibility-psi} holds trivially for
    $\psi(x_0,\ldots,x_h)=|x_h|\ind{|x_0|>s_0}$.
  \item \eqref{eq:psi-second-moment-extrindep} of \Cref{hypo:moments} implies \eqref{eq:psi-second-moment-1-extrindep}.
  \item \eqref{eq:conditiondhs-ext-extrindep} of \Cref{hypo:moments} implies \eqref{eq:condition-sum-phi}.
  \end{itemize}
  For tightness, we apply \Cref{thm:fclt-array}.  Condition~\eqref{eq:second-order-psi} reduces
  to~\eqref{eq:no-bias-sums}. It remains to
  verify~\ref{item:pointwise-separable}-\ref{item:continuity}. The class $\mcg_1$ is separable and
  linearly ordered, hence VC-subgraph class.  The envelope function $|x_h|\ind{|x_0|>s_0}$ belongs
  to $\mcg_1$.

Define the semi-metric $\rho_h$ on $\mcg_1$ by
\begin{align}
  \label{eq:metric-sums}
  \rho_h(\psi_s,\psi_t) = \esp[Y_0^2W_h^2\ind{Y_0\in (s,t]}]  \; .
\end{align}
Since \Cref{hypo:moments} implies  $\alpha>2$ we have for $s<t$,
\begin{align*}
  \rho_h(\psi_s,\psi_t)=\frac{\alpha}{\alpha-2}\esp[W_h^2](s^{-\alpha+2}-t^{-\alpha+2})\leq \constant (t-s)\;.
\end{align*}
Hence, $(\mcg_1,\rho_h)$ is totally bounded. Moreover, by regular variation and the uniform convergence
theorem, the convergence
\begin{align*}
  \lim_{n\to\infty}\frac{1}{b_h^2(u_n)\tail{F}_0(u_n)}\esp[X_h^2\left(\ind{X_0>su_n}   - \ind{X_0>tu_n}\right)]
  = \frac{\alpha}{\alpha-2}\esp[W_h^2](s^{-\alpha+2}-t^{-\alpha+2})
\end{align*}
is uniform on compact sets of $(0,\infty)$. Hence, \ref{item:continuity} holds.  The joint
convergence holds by applying \Cref{theo:tep-extrindep} and considering the class
$\mcg=\mcg_0\cup\mcg_1$.
\end{proof}

\subsection{Proof of \Cref{coro:tepindep-semifeasible,theo:clt-psi}}

\begin{proof}[Proof of \Cref{coro:tepindep-semifeasible}]
  Set $\tepabs_n(s) = k^{-1} \sum_{j=1}^n \ind{|X_j|>\tepseq s}$. By \Cref{cor:sums} and Vervaat's lemma
  \cite{vervaat:1972}, we obtain the joint convergence
  \begin{multline}
    \left\{\tepextrindep_{h,n}(\psi_s),\TEPextrindep_{h,n}(s,y),\sqrt{k}
      \left(\tepabs_n(s)-s^{-\alpha}\right),  \sqrt{k}\left(\tepseq^{-1}\orderstat[|X|]{n}{k}-1\right) \right\} \\
    \convweak
    \left\{\tepextrindep_h(\psi_s),\tepextrindep_h(\phi_{s,y}),\tepextrindep_h(\phi_{s,\infty}),
      \alpha^{-1}\tepextrindep_h(\phi_{1,\infty})\right\} \; .
    \label{eq:joint-convergence}
  \end{multline}
  Write $\xi_n=\tepseq^{-1}\orderstat[|X|]{n}{k}$ and note that the above weak convergence implies
  that $\xi_n\to 1$ in probability.  Note that by~(\ref{eq:nucondpsi_s}), we have
  \begin{align*}
    \nucond_{0,h}(\psi_{\xi_n})         =\xi_n^{\scalingexp_h-\alpha}  \;.
  \end{align*}
  In view of this identity and  \eqref{eq:convergence-sums}, we have
  \begin{align}
    \label{eq:integrals-main-decomposition}
    \sqrt{k} \left(  \frac{\widehat\scalfunccev_{h,n}}{\scalfunccev_{h}(\tepseq)} - 1 \right)
    &  = \tepextrindep_{h,n}(\psi_{\xi_n})+\sqrt{k}\left\{\xi_n^{\scalingexp_h-\alpha}-1\right\}\;.
  \end{align}
  Thus,
  \begin{align*}
    \sqrt{k} \left(  \frac{\widehat\scalfunccev_{h,n}}{\scalfunccev_{h}(\tepseq)} - 1 \right)
    &    \convdistr \tepextrindep_h({\psi_1})  +  \alpha^{-1}(\scalingexp_h-\alpha)  \tepextrindep_h(\phi_{1,\infty})  \; .
  \end{align*}
  The process $\bmotion$ defined by $\bmotion(u) = \tepextrindep_h(\phi_{1,\cevcdfuniv_h^{-1}(u)})$
  is a standard Brownian motion and $\bbridge(u) = \bmotion(u) -u\bmotion(1)$ is a standard Brownian
  bridge. Therefore,
  \begin{align*}
    \tepextrindep_h({\psi_1}) = \int_0^1 |\cevcdfuniv_h^{-1}(u)| \bmotion(\rmd u) \; .
  \end{align*}
  By assumption, we have
  \begin{align*}
    \int_0^1 |\cevcdfuniv_h^{-1}(u)| \rmd u  = \int_{-\infty}^\infty |y| \cevcdfuniv_h(\rmd y) = 1 \; ,
  \end{align*}
  thus $\tepextrindep_h({\psi_1})  -  \tepextrindep_h(\phi_{1,\infty}) = \int_0^1 | \cevcdfuniv_h^{-1}(u)| \bbridge(\rmd u)$.
\end{proof}

\begin{proof}[Proof of \Cref{theo:clt-psi}]
  Writing $\theta_n= \scalfunccev_h^{-1}(\tepseq) \widehat\scalfunccev_{h,n}$, we have
  \begin{align}
    \sqrt{\statinterseq} \left\{\hatcevcdf_{h,n}(y) - \cevcdfuniv_h(y)\right\}
    & = \sqrt{\statinterseq} \left\{\TEDextrindep_{h,n} \left(\frac{\orderstat{n}{n-\statinterseq}}{\tepseq},
      \frac{\widehat\scalfunccev_{h,n}}{\scalfunccev_h(\tepseq)}y\right) -\cevcdfuniv_{h}(y)\right\} \nonumber  \\
    & = \TEPextrindep_{h,n}(\xi_n,\theta_ny) +
      \sqrt{\statinterseq} \left(\xi_n^{-\alpha} \cevcdfuniv_h(\xi_n^{-\scalingexp_h}\theta_n y) - \cevcdfuniv_h(y)\right) \label{eq:mainterms}\;.
  \end{align}
  By \Cref{coro:tepindep-semifeasible}, $\xi_n\convprob 1$ and $\theta_n\convprob 1$. Hence, the
  first term in (\ref{eq:mainterms}) converges to $\tepextrindep_h(\phi_{1,y})$.  The delta method implies
  that the limiting behaviour of the second term in~(\ref{eq:mainterms}) is the same as that of
  \begin{align*}
    \sqrt{k}\cevcdfuniv_h(y)\left\{\xi_n^{-\alpha}-1\right\}
    + \xi_n^{-\alpha}\sqrt{k}y\cevcdfuniv_h'(y)\left\{\xi_n^{-\kappa_h}-1+\xi_n^{-\kappa_h}(\theta_n-1)\right\} \; ,
  \end{align*}
  which by \Cref{coro:tepindep-semifeasible} converges weakly to
  \begin{align*}
    - \cevcdfuniv_h(y) \tepextrindep_h(\phi_{1,\infty}) + y \cevcdfuniv_h'(y) \{
    \tepextrindep_h(\psi_{1}) - \tepextrindep_h(\phi_{1,\infty}) \} \; .
  \end{align*}
  Moreover, with $\bmotion$ and $\bbridge$ as above, we have
  \begin{align*}
    \{    \tepextrindep_h(\phi_{1,y}) - \cevcdfuniv_h(y) \tepextrindep_h(\phi_{1,\infty}) , y\in\Rset\}
    = \bbridge \circ \cevcdfuniv \; .
  \end{align*}
  This concludes the proof.
\end{proof}

\subsection{Proof of~\Cref{thm:clt-kappa}}
Set $\beta_h=\alpha/(1+\kappa_h)$ and $\gamma_h = \beta_h^{-1}$.  At the first step we justify
functional convergence of the tail empirical process based on products:
\begin{align}
  \label{eq:tep-product}
  \sqrt{n\tail{F}_0(u_n)}\left\{\frac{1}{n\tail{F}_0(u_n)} \sum_{j=1}^{n-h}\ind{|X_jX_{j+h}| > \|W_h\|_{\beta_h} u_n b_h(u_n) s }-s^{-\beta_h}\right\} \; .
\end{align}
Define $\psi(x_0,\ldots,x_h)=\ind{|x_0|>s_0}+\ind{\|W_h\|_{\beta_h}|x_0x_h|> s_0}$ and the class
$\mcg=\mcg_0'\cup\mcg_1'$ with
\begin{align*}
  \mcg_0' & = \{\mci_{s}:(x_0,\ldots,x_h)\to \ind{|x_0|>s},s\geq s_0\} \; ,  \\
  \mcg_1' & = \{\varphi_{s}:(x_0,\ldots,x_h)\to \ind{|x_0x_h|> s\|W_h\|_{\beta_h}},s\geq s_0\} \; .
\end{align*}
With this notation the process defined in \eqref{eq:tep-product} can be written as
$\tepextrindep_{h,n}(\varphi_s)$. Similarly, the tail empirical process of $X_j$'s can be written as
$\tepextrindep_{h,n}(\mci_{s})$:
\begin{align*}
\tepextrindep_{h,n}(\mci_{s})=\sqrt{n\tail{F}_0(u_n)}\left\{\frac{1}{n\tail{F}_0(u_n)}\sum_{j=1}^{n}\ind{|X_j|> u_n s }-s^{-\alpha}\right\}\;.
\end{align*}
We note that $\mcg_0'\subseteq \mcg_0$, where $\mcg_0$ was defined in the proof of
\Cref{theo:tep-extrindep} and $\mci_s = \phi_{s,\infty}$.
Assumptions~(\ref{eq:condition-epsilon-product-cev-better}) and~(\ref{eq:conditiondhs-product})
imply~(\ref{eq:negligibility-psi}) and~(\ref{eq:condition-sum-phi}) for the class $\mcg_1'$.  The
bias condition~(\ref{eq:second-order-psi}) is implied by
(\ref{eq:no-bias-orderstat-bivariate-product}).  Therefore, by \Cref{lem:clt-array-fidi}, the finite
dimensional distributions of $\tepextrindep_{h,n}$ converge to those of $\tepextrindep_{h}$ of
$\mcg_1'$.
Moreover, the envelope function of $\mcg$ is $\psi$. Furthermore,
\begin{itemize}
\item The class $\mcg$ is the union of two linearly ordered classes, hence is a VC subgraph class.
\item We consider the metric $\rho_h$ on $\mcg$ induced by the covariance, that is
  for~$\phi,\varphi\in \mcg$,
  \begin{align*}
    \rho_h(\phi,\varphi) = \sqrt{ \var(\tepextrindep_h(\phi)-\tepextrindep_h(\varphi)) } \;.
  \end{align*}
  The metric $\rho_h$ restricted to $\mcg_0'$ and $\mcg_1'$ respectively becomes
  \begin{align*}
    \rho_h(\mci_s,\mci_t) = \sqrt{ |s^{-\alpha} - t^{-\alpha}|}\;, \ \
    \rho_h(\varphi_s,\varphi_t) & = \sqrt{ |s^{-\beta_h} - t^{-\beta_h}|}  \; .
  \end{align*}
  Thus it easily seen that both $\mcg_0'$ and $\mcg_1'$ are totally bounded for the metric
  $\rho_h$. This means that Condition~\ref{item:totally-bounded} holds.
\item By regular variation, the convergence
  \begin{align*}
    \lim_{n\to\infty}\frac{1}{\tail{F}_0(\tepseq)}\esp[\ind{|X_0X_h|\in \tepseq b_h(\tepseq)(s,t]}]=\rho_h^2(\varphi_s,\varphi_t)
  \end{align*}
  is uniform \wrt\ $s,t\in[s_0,\infty)$ (see
  \cite[Theorem~1.5.2]{bingham:goldie:teugels:1989}). Therefore,~\ref{item:continuity} holds on
  $\mcg_1'$.
\end{itemize}
Thus we have proved that
\begin{align*}
  \tepextrindep_{h,n} \implies \tepextrindep_h \; , \ \mbox{ in } \ell^\infty(\mcg) \; .
\end{align*}
Recall that $\xi_n=\orderstat[|X|]{n}{n-k}/\tepseq$. Set $V_j=|X_jX_{j+h}|$ and
$v_n = \tepseq b(\tepseq)\|W_h\|_{\beta_h}$. Define
\begin{align*}
  \zeta_{n} = \orderstat[V]{n}{n-k} / v_n \; .
\end{align*}
As a consequence of the functional convergence and Vervaat's Lemma, we obtain
\begin{align}
  \label{eq:joint-quantile}
  \left(\tepextrindep_{h,n},  \sqrt{k} (\xi_n-1), \sqrt{k}(\zeta_n-1) \right)
  \implies (\tepextrindep_h,\gamma\tepextrindep_h(\mci_1),\gamma_h\tepextrindep_h(\varphi_1)) \; ;
\end{align}
By elementary algebra and Fubini's theorem, we have
\begin{align*}
  \sqrt{k} (\hat\gamma-\gamma)
  & = \int_1^\infty \tepextrindep_{h,n} (\mci_{s}) \frac{\rmd s}s
    + \int_{\xi_n}^1\tepextrindep_{h,n} (\mci_{s}) \frac{\rmd s}s
    + \gamma \sqrt{k}(\xi_n^{-1/\gamma} -1) \; , \\
  \sqrt{k} (\hat\gamma_h-\gamma_h)
  & = \int_1^\infty \tepextrindep_{h,n} (\varphi_{s}) \frac{\rmd s}s
    + \int_{\zeta_n}^1 \tepextrindep_{h,n} (\varphi_{s}) \frac{\rmd s}s
    + \gamma_h \sqrt{k}(\zeta_n^{-1/\gamma_h} - 1) \; .
\end{align*}
The middle term of both expansions is $o_P(1)$.  The last terms are dealt with by Slutsky's theorem
and~(\ref{eq:joint-quantile}):
\begin{align}
  \label{eq:corrections}
  \left(  \sqrt{k}(\xi_n^{-1/\gamma} - 1) , \sqrt{k}(\zeta_n^{-1/\gamma_h} - 1)\right) \implies
  \left( -\tepextrindep_h(\mci_1),-\tepextrindep_h(\varphi_1)\right) \; .
\end{align}
Furthermore,  jointly with the previous convergences,
\begin{align}
  \label{eq:integrals}
  \left(\int_1^\infty \tepextrindep_{h,n} (\mci_{s}) \frac{\rmd s}s, \int_1^\infty \tepextrindep_{h,n} (\varphi_{s}) \frac{\rmd s}s  \right)
  \implies   \left(\int_1^\infty \tepextrindep_{h} (\mci_{s}) \frac{\rmd s}s, \int_1^\infty \tepextrindep_{h} (\varphi_{s}) \frac{\rmd s}s\right) \; .
\end{align}
Indeed, for a fixed $A>1$, we can write
\begin{align*}
  \int_1^\infty \tepextrindep_{h,n} (\mci_{s}) \frac{\rmd s}s =
  \int_1^A \tepextrindep_{h,n} (\mci_{s}) \frac{\rmd s}s +   \int_A^\infty \tepextrindep_{h,n} (\mci_{s}) \frac{\rmd s}s \; , \\
  \int_1^\infty \tepextrindep_{h,n} (\varphi_{s}) \frac{\rmd s}s =
  \int_1^A \tepextrindep_{h,n} (\varphi_{s}) \frac{\rmd s}s +   \int_A^\infty \tepextrindep_{h,n} (\varphi_{s}) \frac{\rmd s}s \; .
\end{align*}
By \Cref{thm:fclt-array}, we have
\begin{align*}
  \left(  \int_1^A \tepextrindep_{h,n} (\mci_{s}) \frac{\rmd s}s ,  \int_1^A \tepextrindep_{h,n} (\varphi_{s}) \frac{\rmd s}s \right) \implies
  \left( \int_1^A \tepextrindep_h (\mci_{s}) \frac{\rmd s}s  ,   \int_1^A \tepextrindep_h (\varphi_{s}) \frac{\rmd s}s \right) \; .
\end{align*}
Moreover, as $A\to\infty$,
\begin{align*}
  \left(  \int_1^A \tepextrindep_{h} (\mci_{s}) \frac{\rmd s}s ,  \int_1^A \tepextrindep_{h} (\varphi_{s}) \frac{\rmd s}s \right) \convprob
  \left( \int_1^\infty \tepextrindep_{h} (\mci_{s}) \frac{\rmd s}s  ,   \int_1^\infty \tepextrindep_h (\varphi_{s}) \frac{\rmd s}s \right) \; .
\end{align*}
Therefore, we must prove that for all $\epsilon>0$,
\begin{align*}
  \lim_{A\to\infty} \limsup_{n\to\infty} \pr \left( \left|  \int_A^\infty \tepextrindep_{h,n} (\mci_{s}) \frac{\rmd s}s \; . \right| >\epsilon \right)
  = 0 \; , \\
  \lim_{A\to\infty} \limsup_{n\to\infty} \pr \left( \left|  \int_A^\infty \tepextrindep_{h,n} (\varphi_{s}) \frac{\rmd s}s \; . \right| >\epsilon \right)
  = 0 \; .
\end{align*}
The proof of these bounds rely on the $\beta$ mixing property and
Conditions~(\ref{eq:condition-slog-univ}) and~(\ref{eq:condition-slog-product}). We will only prove
the first one, the second being exactly similar. First note that
\begin{align*}
  \int_{A}^\infty \tepabs_n(\mci_s) \frac{\rmd s}s = \frac1k \sum_{j=1}^n \log_+(|X_j|/A\tepseq)
\end{align*}
In view of \Cref{hypo:basics}, it suffices to consider independent blocks and thus to compute the
variance of one block, that is we must prove that
\begin{align*}
  \lim_{A\to\infty} \limsup_{n\to\infty}  m_n\var\left( k^{-1/2}\sum_{i=1}^{r_n} \log_+(|X_j|/A\tepseq) \right)  = 0
\end{align*}
Let the variance term in the previous display be denoted by $\mcv_n(A)$. Then,
\begin{align*}
  \mcv_n(A) \leq \frac{\esp[\log_+^2(|X_0|/\tepseq)}{\tail{F}_0(\tepseq)}
  + 2 \sum_{j=1}^{r_n} \frac{\esp[\log_+(|X_0|/A\tepseq)\log_+(|X_j|/A\tepseq)]}{\tail{F}_0(\tepseq)} \; .
\end{align*}
By regular variation, extremal independence and condition~(\ref{eq:condition-slog-univ}), for every
$\epsilon>0$, we can choose an integer $m$ such that
\begin{align*}
  \limsup_{n\to\infty} \mcv_n(A)
  & = \int_A^\infty \log^2(x/A) \alpha x^{-\alpha-1} \rmd x \\
  &  \phantom{ = } + 2 \limsup_{n\to\infty}  \sum_{j=m+1}^{r_n} \frac{\esp[\log_+(|X_0|/\tepseq)\log_+(|X_j|/\tepseq)]}{\tail{F}_0(\tepseq)} \\
  & \leq A^{-\alpha} \int_1^\infty \log^2(y) \alpha y^{-\alpha-1} \rmd y + \epsilon \; .
\end{align*}
Thus $\lim_{A\to\infty} \limsup_{n\to\infty} \mcv_n(A) \leq \epsilon$. Since $\epsilon$ is
arbitrary, this proves the requested bound.

 Combining~(\ref{eq:corrections}) and~(\ref{eq:integrals}), we have proved that
\begin{align}
  \label{eq:joint-convergence-hill}
  \sqrt{k}\left( \hat\gamma-\gamma,\hat\gamma_h-\gamma_h\right) \convdistr \left( \int_1^\infty
    \tepextrindep_h(\mci_s)\frac{\rmd s}s - \tepextrindep_h(\mci_1), \int_1^\infty
    \tepextrindep_h(\psi_s)\frac{\rmd s}s - \tepextrindep_h(\psi_1) \right) \; .
\end{align}
Define the Gaussian process $\mbg$ on $(0,\infty)\times\Rset$ by
$\mbg(s,y) = \tepextrindep_h(\phi_{s,y})$. Then,
\begin{align*}
  \int_1^\infty  \tepextrindep_h(\mci_s)\frac{\rmd s}s - \tepextrindep_h(\mci_1)
  & = \int_0^\infty \int_{-\infty}^\infty \{\log_{+}(u)-\gamma\ind{u>1}\}    \mbg(\rmd u\rmd v)  \; ,  \\
  \int_1^\infty  \tepextrindep_h(\varphi_s)\frac{\rmd s}s - \mbg(\varphi_1)
  &  =    \int_0^\infty \int_{-\infty}^\infty \{\log_{+}(u|v|) - \gamma_h \ind{u|v|>1} \}  \mbg(\rmd u\rmd v) \; .
\end{align*}
Write
\begin{align*}
  \sqrt{k}(\hat\scalingexp_h-\scalingexp)
  = \sqrt{k}\left\{\widehat\gamma_h/\widehat\gamma-\gamma_h/\gamma\right\}=\sqrt{k}\frac{1}{\widehat\gamma}\left\{\widehat\gamma_h-\gamma_h\right\}
  -  \sqrt{k}\frac{\gamma_h}{\hat\gamma\gamma}\left\{\widehat\gamma^{-1}-\gamma^{-1}\right\} \; .
\end{align*}
Applying the joint convergence~(\ref{eq:joint-convergence-hill}), Slutsky lemma and the delta
method, we obtain,
\begin{multline*}
  \sqrt{k}(\hat\kappa_h-\kappa_h)   \convdistr \\
  (1+\kappa_h) \int_0^\infty\!\!\int_{-\infty}^\infty \left(\beta_h\log_{+}(u|v|)-\ind{u|v|>1}
    -\alpha\log_{+}(u)+\ind{u>1}\right) \mbg(\rmd u,\rmd v) \; .
\end{multline*}
There remains to compute the variance $\sigma^2$ of the limiting distribution which is Gaussian with
zero mean. Setting
\begin{align*}
  h(u,v) = \beta_h\log_{+}(u|v|)-\ind{u|v|>1} - \alpha\log_{+}(u) + \ind{u>1} \; ,
\end{align*}
and combining with~\eqref{eq:representation-nucond} we obtain
\begin{align*}
  \sigma^2 & = (1+\kappa_h)^2\esp\left[\int_0^\infty h^2(s,s^{\kappa_h }|W_h|)\alpha s^{-\alpha-1}\rmd s\right]\;.
\end{align*}
Since  $h^2(s,s^{\kappa_h} |W_h|)=0$ if $s<1$ and $s^{1+\kappa_h}|W_h|<1$, we have
\begin{align*}
  \sigma^2&=(1+\kappa_h)^2\esp\left[\int_{1\wedge W_h^{-1/(1+\kappa_h)}}^\infty h^2(s,s^{\kappa_h}|W_h|)\alpha s^{-\alpha-1}\rmd s\right]\;.
\end{align*}
Setting $A=|W_h|^{\beta_h}$ and substituting $t=s^{-\alpha}$ we obtain
\begin{align*}
  &\int_{1\wedge W_h^{-1/(1+\kappa_h)}}^\infty h^2(s,s^{\kappa_h}|W_h|)\alpha s^{-\alpha-1}\rmd s=
    \int_0^{1\vee A}h^2(t^{-1/\alpha},t^{-\kappa_h/\alpha}|W_h|)\rmd t\;.
\end{align*}
We also note that
\begin{align*}
  h(t^{-1/\alpha},t^{-\kappa_h/\alpha}|W_h|)=\left\{\log_+(t^{-1}A)-\ind{A>t}\right\}-\left\{\log_+(t^{-1})-\ind{t<1}\right\}\;.
\end{align*}
If $A<1$, then we evaluate
\begin{align*}
  \int_0^{1}h^2(t^{-1/\alpha},t^{-\kappa_h/\alpha}|W_h|)\rmd t=A\log^2(A)+\int_A^1\left\{\log (t)+1\right\}^2\rmd t\;.
\end{align*}
Since the primitive of $\left\{\log (t)+1\right\}^2$ is $t\log^2(t)+t$, the right hand side becomes
$1-A$.  If $A>1$, then
\begin{align*}
  \int_0^{1}h^2(t^{-1/\alpha},t^{-\kappa_h/\alpha}|W_h|)\rmd t=\log^2(A)+\int_1^A\left\{\log(A)-\log (t)-1\right\}^2\rmd t=A-1\;.
\end{align*}
Summarizing,
\begin{align*}
  \sigma^2=(1+\kappa_h)^2\esp[|1-|W_h|^{\beta_h}|]\;.
\end{align*}

\paragraph{Acknowledgements} The research of Clemonell Bilayi and Rafal Kulik was supported by the
NSERC grant 210532-170699-2001.  The research of Philippe Soulier was partially supported by the
LABEX MME-DII.

\bigskip
{\small
  \noindent Clemonell Bilayi-Biakana: {\tt cbilayib@uottawa.ca}\\
  University of Ottawa,\\
  Department of Mathematics and Statistics,\\
  585 King Edward Av.,\\
  K1J 8J2 Ottawa, ON, Canada }

\bigskip
{\small
  \noindent Rafa{\l}~Kulik: {\tt rkulik@uottawa.ca}\\
  University of Ottawa,\\
  Department of Mathematics and Statistics,\\
  585 King Edward Av.,\\
  K1J 8J2 Ottawa, ON, Canada }

\bigskip
{\small
  \noindent Philippe Soulier: {\tt philippe.soulier@u-paris10.fr}\\
  Universit\'e  Paris Nanterre,\\
  D\'epartement de math\'ematiques et informatique,\\
  Laboratoire MODAL'X EA3454,\\
  92000 Nanterre, France }


\begin{thebibliography}{vdVW96}

\bibitem[BGT89]{bingham:goldie:teugels:1989}
Nicholas~H. Bingham, Charles~M. Goldie, and Jozef~L. Teugels.
\newblock {\em Regular variation}.
\newblock Cambridge University Press, Cambridge, 1989.

\bibitem[Bra05]{bradley:2005}
Richard~C. Bradley.
\newblock Basic properties of strong mixing conditions. {A} survey and some
  open questions.
\newblock {\em Probability Surveys}, 2:107--144, 2005.

\bibitem[DJe17]{drees:janssen:2017}
Holger Drees and Anja Jan\ss~en.
\newblock Conditional extreme value models: fallacies and pitfalls.
\newblock {\em Extremes}, 20(4):777--805, 2017.

\bibitem[DM01]{davis:mikosch:2001}
Richard~A. Davis and Thomas Mikosch.
\newblock Point process convergence of stochastic volatility processes with
  application to sample autocorrelation.
\newblock {\em Journal of Applied Probability}, 38A:93--104, 2001.
\newblock Probability, statistics and seismology.

\bibitem[DR10]{drees:rootzen:2010}
Holger Drees and Holger Rootz{\'e}n.
\newblock Limit theorems for empirical processes of cluster functionals.
\newblock {\em Ann. Statist.}, 38(4):2145--2186, 2010.

\bibitem[DR11]{das:resnick:2011}
Bikramjit Das and Sidney~I. Resnick.
\newblock Conditioning on an extreme component: model consistency with regular
  variation on cones.
\newblock {\em Bernoulli}, 17(1):226--252, 2011.

\bibitem[Dre00]{drees:2000}
Holger Drees.
\newblock Weighted approximations of tail processes for {$\beta$}-mixing random
  variables.
\newblock {\em The Annals of Applied Probability}, 10(4):1274--1301, 2000.

\bibitem[Dre02]{drees:2002tail}
Holger Drees.
\newblock Tail empirical processes under mixing conditions.
\newblock In {\em Empirical process techniques for dependent data}, pages
  325--342. Birkh\"auser Boston, Boston, MA, 2002.

\bibitem[DSW15]{drees:segers:warchol:2015}
Holger Drees, Johan Segers, and Micha{\l} Warcho{\l}.
\newblock Statistics for tail processes of {M}arkov chains.
\newblock {\em Extremes}, 18(3):369--402, 2015.

\bibitem[HL06]{hult:lindskog:2006}
Henrik Hult and Filip Lindskog.
\newblock Regular variation for measures on metric spaces.
\newblock {\em Publ. Inst. Math. (Beograd) (N.S.)}, 80(94):121--140, 2006.

\bibitem[HR07]{heffernan:resnick:2007}
Janet~E. Heffernan and Sidney~I. Resnick.
\newblock Limit laws for random vectors with an extreme component.
\newblock {\em The Annals of Applied Probability}, 17(2):537--571, 2007.

\bibitem[JD16]{janssen:drees:2016}
Anja Janssen and Holger Drees.
\newblock A stochastic volatility model with flexible extremal dependence
  structure.
\newblock {\em Bernoulli}, 22(3):1448--1490, 2016.

\bibitem[Kal17]{kallenberg:2017}
Olav Kallenberg.
\newblock {\em Random Measures, Theory and Applications}, volume~77 of {\em
  Probability Theory and Stochastic Modelling}.
\newblock Springer-Verlag, New York, 2017.

\bibitem[KS11]{kulik:soulier:2011}
Rafa{\l} Kulik and Philippe Soulier.
\newblock The tail empirical process for long memory stochastic volatility
  sequences.
\newblock {\em Stochastic Processes and their Applications}, 121(1):109 -- 134,
  2011.

\bibitem[KS15]{kulik:soulier:2015}
Rafa{\l} Kulik and Philippe Soulier.
\newblock Heavy tailed time series with extremal independence.
\newblock {\em Extremes}, 18:273--299, 2015.

\bibitem[KSW15]{kulik:soulier:wintenberger:2015}
Rafa{\l} Kulik, Philippe Soulier, and Olivier Wintenberger.
\newblock Practical conditions for weak convergence of tail empirical
  processes.
\newblock arxiv:1511.04903, 2015.

\bibitem[LRR14]{lindskog:resnick:roy:2014}
Filip Lindskog, Sidney~I. Resnick, and Joyjit Roy.
\newblock Regularly varying measures on metric spaces: hidden regular variation
  and hidden jumps.
\newblock {\em Probability Surveys}, 11:270--314, 2014.

\bibitem[MR13]{mikosch:rezapour:2013}
Thomas Mikosch and Mohsen Rezapour.
\newblock Stochastic volatility models with possible extremal clustering.
\newblock {\em Bernoulli}, 19(5A):1688--1713, 2013.

\bibitem[Roo09]{rootzen:2009}
Holger Rootz\'{e}n.
\newblock Weak convergence of the tail empirical process for dependent
  sequences.
\newblock {\em Stoch. Proc. Appl.}, 119(2):468--490, 2009.

\bibitem[vdVW96]{vandervaart:wellner:1996}
Aad~W. van~der Vaart and Jon~A. Wellner.
\newblock {\em Weak convergence and empirical processes}.
\newblock Springer, New York, 1996.

\bibitem[Ver72]{vervaat:1972}
Wim Vervaat.
\newblock Functional central limit theorems for processes with positive drift
  and their inverses.
\newblock {\em Z. Wahrscheinlichkeitstheorie und Verw. Gebiete}, 23:245--253,
  1972.

\end{thebibliography}
\end{document}